\documentclass[sn-mathphys,Numbered]{sn-jnl}

\usepackage[normalem]{ulem}
\usepackage{graphicx}%
\usepackage{multirow}%
\usepackage{amsmath,amssymb,amsfonts}%
\usepackage{amsthm}%
\usepackage{mathrsfs}%
\usepackage[title]{appendix}%
\usepackage{xcolor}%
\usepackage{textcomp}%
\usepackage{manyfoot}%
\usepackage{booktabs}%
\usepackage{algorithm}%
\usepackage{algorithmicx}%
\usepackage{algpseudocode}%
\usepackage{listings}%
\usepackage[pagewise,displaymath]{lineno}
\usepackage{hyperref}
\newtheorem{theorem}{Theorem}
\newtheorem{proposition}[theorem]{Proposition}%
\newtheorem{example}{Example}%
\newtheorem{definition}{Definition}%
\usepackage{subcaption}
\newcommand{\bnodes}{$\bullet$-nodes }
\newcommand{\cnodes}{$\times$-nodes }
\newcommand{\SL}{{\em SL}}
\newcommand{\UB}{{\em UB}}

\newcommand{\Dx}{\Delta x}				
\newcommand{\Dt}{\Delta t}

\title{Coupled Scheme for Linear and Hamilton-Jacobi Equations: Theoretical and Numerical Aspects}
\author*[1]{\fnm{Smita} \sur{Sahu}}\email{smita.sahu@port.ac.uk}

\affil*[1]{\orgdiv{School of Mathematics and Physics}, \orgname{University of Portsmouth}, \orgaddress{\street{Lion Terrace}, \city{Portsmouth}, \postcode{PO1 3HF}, \state{Hampshire}, \country{UK}}}
\begin{document}
\def\vet#1{{{\boldsymbol{#1}}}}
\maketitle
\begin{abstract}

We present a comprehensive analysis of the coupled scheme introduced in [Springer Proceedings in Mathematics \& Statistics, vol 237. Springer, Cham 2018 \cite{S2018}] for linear and Hamilton-Jacobi equations. This method merges two distinct schemes, each tailored to handle specific solution characteristics. It offers a versatile framework for coupling various schemes, enabling the integration of accurate methods for smooth solutions and the treatment of discontinuities and gradient jumps. In \cite{S2018}, the emphasis was on coupling an anti-dissipative scheme designed for discontinuous solutions with a semi-Lagrangian scheme developed for smooth solutions. In this paper, we rigorously establish the essential properties of the resulting coupled scheme, especially in the linear case. To illustrate the effectiveness of this coupled approach, we present a series of one-dimensional examples.
\end{abstract}
\section{Introduction}\label{sec:intro}
In this paper we aim to prove some of the properties of coupled scheme proposed in \cite{S2018} for first order time dependent Hamilton-Jacobi (HJ) equations. We consider the following one-dimensional Cauchy problem
\begin{eqnarray}\label{eq:hj0}
\left\{
\begin{array}{ll}
\partial_t u+H(x,D u)=0, & (x,t)\in \mathbb{R}\times [0,T],\\
u(x,0)=u_0(x),& x\in \mathbb{R},
\end{array}
\right.
\end{eqnarray}
\noindent
where the Hamiltonian $H$ is convex in the gradient.
A classical motivation comes from optimal control theory where $H(x,\nabla u)=\max\limits_{\alpha\in A}\{f(x,\alpha)u_{x}(t,x)\}$ and $\alpha$ represents the control. It is well known that in this framework the solution $u$ of~\eqref{eq:hj0} corresponds to the value function of the corresponding control problem~\cite{BCD97, B94}. Typically, solutions are Lipschitz continuous when the data are Lipschitz continuous. However, in various applications such as control problems with state constraints, games, and image processing, discontinuous solutions are encountered. 

In \cite{S2018}, a technique  was developed that combines the semi-Lagrangian (\SL)~\cite{FF14} and ultra-bee (\UB) \cite {BZ07} scheme to solve both advection problems and HJ equations. Typically, the two initial schemes exhibit distinct characteristics, with one excelling in smooth regions of the solution and the other being more adept at handling discontinuities. The concept behind the coupling scheme is to create a new method that combines the strengths of both schemes without introducing excessive computational costs.
This coupling approach is inspired by hybrid schemes used in hyperbolic conservation laws. The choice between methods in the coupled scheme depends on a regularity indicator, incurring a small additional computational cost. It's worth noting that similar couplings are possible, particularly if the schemes use the same nodes. In one dimension, HJ equations are linked to hyperbolic conservation laws, with the viscosity solution of the HJ equation being the primitive of the entropy solution of the corresponding hyperbolic conservation law. Various numerical schemes have been developed for hyperbolic conservation laws (see e.g.~\cite{H83,H84,got-shu-98}, and many of these ideas extend to HJ equations. 
We also mention that more recently a new class of high-order filtered schemes has been proposed~\cite{BFS15} and improved \cite{F2020}, these schemes converge to the viscosity solution and a precise error estimate has been proved. It can be interesting to deal with discontinuous viscosity solutions so these schemes have to be adapted in order to obtain reasonable approximations which do not diffuse too much around the discontinuities of $Du$ and/or $u$ and do not introduce spurious oscillations. 

\noindent
In this paper, we recall the coupled scheme for~\eqref{eq:hj0} from \cite{S2018} based on the coupling between the \UB~ scheme (a particular anti-dissipative scheme) and a first order \SL~scheme. Idea is to take the advantage of the properties of the two methods introducing an indicator parameter $\sigma^n_j$ which will be computed in every cell $C_j$ at every time step in order to detect if there is a singularity or a jump discontinuity there. Then, according to the value of the indicator $\sigma^n_j$, we will use the \SL~scheme if the solution is regular enough switching to the \UB~scheme when a discontinuity is detected.  It has been discuss in \cite{S2018} that one of the difficulties in this coupling is that they use different grids and different values: the \SL~scheme computes approximate values at the nodes whereas the \UB~scheme typically works on averaged values which are cell centred, so we have to introduce some projection operators on the grids to switch from one scheme to the other. For schemes working on the same grid and with the same kind of approximate values this projection is not necessary.

\emph{Organisation of the paper.} In \S\ref{sec:background}, we will recall the \SL~\cite{FF14} and \UB~\cite{BZ07} schemes and the coupled scheme from \cite{S2018}. In \S\ref{sec:prop}, we will prove some important properties of the coupled \SL~+\UB~scheme when applied to the linear advection equation. Finally \S\ref{sec:cs_numtests}, will be devoted to the analysis of some numerical tests in one-dimension for the linear and the nonlinear equation.


\section{Background results for the uncoupled schemes.}\label{sec:background}
To make it easier for the reader, we have recall a concise summary of the \SL schemes by Falcone and Ferretti from~\cite{FF94}, the \UB scheme from~\cite{BZ07}, and the coupling scheme introduced in~\cite{S2018}. Our notation aligns with that used in~\cite{S2018}.

\subsection{Semi-Lagrangian Schemes (SL)~\cite{FF14}} 
In the HJ framework \SL~scheme have been developed initially for the solution of Bellman equations associated with optimal control problems and they can also be interpreted as a discretisation of the dynamic programming principle

In the particular case where the Hamiltonian just depends on the gradient of the solution, i.e. $H(x,Du)=H(Du)$ we have a representation formula for the solution of the Cauchy problem
\begin{equation}
\label{eq:hj1}
\left\{
\begin{array}{ll}
\partial_t u+H(D u)=0, &  (x, t)\in \mathbb{R}\times [0,T]\\
u(x,0)=u_0(x),&x\in \mathbb{R}.
\end{array}
\right.
\end{equation}
provided $H$ satisfies (A2) (the so called {\em Hopf-Lax representation formula}). The \SL~approximation has a strong link with the representation formula, in fact the time discretization can be written as 
$$u(x,t+\Delta t)= \min_{a\in \mathbb{R}}\{u(x-a\Delta t,t)+ \Delta t H^{*}(a)\}$$
where 
$$H^{*}(a)=\sup_{p\in\mathbb{R}} \{a\cdot p-H(p)\}$$
is the {\em Legendre transform} of Hamiltonian $H$. To get the fully discrete version of the scheme one has to introduce a space discretisation. Let us denote by $I_1[v]$  the $P_1$-interpolation (linear interpolation) of a function $v$ in dimension one on the grid $G=\{x_j\}$, i.e. define
\begin{equation}\label{interpolation}
I_1[v](x) = \frac{x_{j+1}-x}{\Dx} v_j+\frac{x-x_{j}}{\Dx} v_{j+1}\; \hbox{ for }  x\in[x_j,x_{j+1}] 
\end{equation}
The \SL~scheme with $P_1$ interpolation corresponding to \eqref{eq:hj1} is
\begin{equation}\label{eq:SL}
u_{j}^{n+1}= \min_{a\in \mathbb{R}}\{ I[u^{n}](x_{j}-a\Delta t)+ \Delta t H^{*}(a)\}
\end{equation}
This scheme is monotone, $L^\infty$-stable and works for the large Courant number.  Moreover,  convergence and error estimates have been proved (the interested reader can find in  \cite{FF14} a detailed presentation of the theory).  
\subsection{Ultra-bee (\UB) scheme for HJ equations}
In this section, we recall the \UB~scheme for the HJ equation from \cite{BZ07}. The \UB~scheme is non-monotone and for the transport equation with constant velocity has an interesting property: it is exact for the class of step functions. We consider the following Cauchy problem
\begin{eqnarray}\label{eq:hj2}
\left\{
\begin{array}{ll}
\partial_t u+ \max\limits_{\alpha \in \mathcal{A}} \{f(x,\alpha)u_x\}=0, \quad  (t,x)\in[0,T]\times \mathbb{R}\\
u(x,0)=u_0(x),\quad x\in \mathbb{R},
\end{array}
\right.
\end{eqnarray}
where $f_m:=\min\limits_{\alpha\in\mathcal{A}}  \{f(x,\alpha)\}$  and  $f_M:=\max\limits_{\alpha\in\mathcal{A}} \{f(x,\alpha)\}$ equation \eqref{eq:hj2} can be written as
\begin{equation}\label{eq:hj3}
\left\{
\begin{array}{ll}
\partial_t u+ \max\{f_m(x)u_x, f_M(x)u_x\}=0, \quad  (t,x)\in[0,T]\times \mathbb{R}\\
u(x,0)=u_0(x),\quad x\in \mathbb{R}.
\end{array}
\right.
\end{equation}
Let $\Delta t$ be a constant time step  and $t_n=n\Delta t$ for $n\geq 0$. Given two velocity functions $f_g:\mathbb{R}\rightarrow \mathbb{R}$, where $g=m,M$,  we introduce the following notation for the corresponding CFL numbers at a node $x_j$, $j\in \mathbb{Z}$:
\begin{equation}
\nu^m_j:= \frac{\Delta t}{\Delta x}f_m(x_j)~\text{and}~ \nu^M_j:= \frac{\Delta t}{\Delta x}f_M(x_j), 
\end{equation}
Then we can define the infinite vectors, $\nu^m=\{\nu^m_j\} _{j\in \mathbb{Z}}$, $\nu^M:=\{\nu^M_j\}_{ j\in \mathbb{Z}}$.  Now let us define the exact cell average values of the approximate solution at time $t_n$ as
\begin{equation}\label{average}
\overline{u}_{j}^{n}=\frac{1}{\Delta x}\int_{x_{j-1/2}}^{x_{j+1/2}} u(x, t_{n}) dx~,j\in \mathbb{Z},~ n\in \mathbb{N}.
\end{equation}
The \UB~scheme will work on this average values (whereas the \SL~scheme works on point-wise values) typically located at the cell center..
Let $\Vert f\Vert_\infty$ denote the $L^\infty$-norm of a bounded function defined on $\mathbb{R}$ 
the CFL condition is
\begin{equation}\label{CFL_fix}
\max\left (\Vert f_m\Vert_\infty,\Vert f_M\Vert_\infty\right)\frac{\Delta t}{\Delta x} \leq 1.
\end{equation}
We recall the algorithm for UB scheme  \ref{UBalgo} from \cite{BZ07,S2018}.
{\small
\begin{algorithm}
\caption{Algorithm for UB scheme}\label{UBalgo}	
\noindent\makebox[\linewidth]{\rule{\textwidth}{0.4pt}}
{\bf{Initialisation:}} Compute the initial averages ${\{\overline{u}_j^0}\}_{j\in \mathbb{Z}}$  as in equation \eqref{average} for $n=0$\\ 
{\bf Main cycle:}  For $n \geq 0$,  compute $\overline u^{n+1}=\{\overline u_j^{n+1}\}_{j\in \mathbb{Z}}$ in the following way:
\begin{algorithmic}
\State	{\bf Step 1.}  For every $j\in \mathbb{Z}$, we define the ``fluxes" $u^n_{j\pm 1/2}(\nu_j)$ for $\nu_j\in\{ \nu^m_j,\nu^M_j\}$ as follows: for $\nu_j \geq 0$, we define
\begin{equation}\label{linEquGrad1}
u^{n,L}_{j+1/2}(\nu_j):=
\left\{
\begin{array}{ll} 
\min\left(\max\left(\overline{u}^n_{j+1}, b^+_j(\nu_j)\right), B^+_j\right)& \text{if}~ \nu_j >0  \\
\overline{u}_{j+1}^n &\text{if}~ \nu_j=0~\text{and}\quad  \overline{u}_j^n\neq \overline{u}^n_{j-1}\\
\overline{u}_{j}^n &\text{if}~ \nu_j=0~\text{and}\quad  \overline{u}_j^n= \overline{u}^n_{j-1}
\end{array}
\right.
\end{equation}
where
\begin{equation}\label{linEquGrad2}
\left\{
\begin{array}{ll} 
b_j^+(\nu_j):= \max\left(\overline{u}_j^n, \overline{u}_{j-1}^n \right)+ \frac{1}{\nu_j}\left( \overline{u}_j^n-\max\left( \overline{u}_j^n, \overline{u}_{j-1}^n\right) \right),\\
B_j^+(\nu_j):= \min\left(\overline{u}_j^n, \overline{u}_{j-1}^n\right)+ \frac{1}{\nu_j}\left( \overline{u}_j^n-\min\left( \overline{u}_j^n, \overline{u}_{j-1}^n\right)\right),\\
\end{array}
\right.
\end{equation}
for $\nu_j <0$, we define
\begin{equation}\label{linEquGrad3}
u^{n,R}_{j-1/2}(\nu_j):=
\left\{
\begin{array}{ll} 
\min\left(\max\left( \overline{u}^n_{j-1}, b^-_j(\nu_j)\right), B^-_j\right)& \text{if}~ \nu_j <0  \\
\overline{u}_{j-1}^n &\text{if}~ \nu_j=0~\text{and}\quad  \overline{u}_j^n\neq \overline{u}^n_{j+1}\\
\overline{u}_{j}^n &\text{if}~ \nu_j=0~\text{and}\quad  \overline{u}_j^n= \overline{u}^n_{j+1}
\end{array}
\right.
\end{equation}
where
\begin{equation}\label{linEquGrad4}
\left\{
\begin{array}{ll} 
b_j^-(\nu_j):= \max\left(\overline{u}_j^n, \overline{u}_{j+1}^n \right)+ \frac{1}{\nu_j}\left( \overline{u}_j^n-\max\left( \overline{u}_j^n, \overline{u}_{j+1}^n\right) \right),\\
B_j^-(\nu_j):= \min\left(\overline{u}_j^n, \overline{u}_{j+1}^n\right)+ \frac{1}{\nu_j}\left( \overline{u}_j^n-\min\left( \overline{u}_j^n, \overline{u}_{j+1}^n\right)\right),\\
\end{array}
\right.
\end{equation}
\State   {\bf Step 2.} For $\nu_j \in\left\{ \nu^m_j, \nu^M_j\right\}$, we define the flux form of the \UB~scheme
\begin{equation}\label{eq:UB1}
\overline{u}^{n+1}_j=\overline{u}^n_j- \nu_j\left({u}^{n,L}_{j+1/2}(\nu)-{u}^{n,R}_{j-1/2}(\nu)\right)
\end{equation}
\State {\bf Step 3.}  Finally, we set $\overline{u}_j^{n+1}:= \min\left(\overline{u}_j^{n+1}(\nu^m_j),\overline{u}_j^{n+1}(\nu^M_j)\right)$, $j\in \mathbb{Z}$.
\vspace{0.3cm}
\end{algorithmic}
\end{algorithm}
}
Note that one can also use  the following short representation 
\begin{equation}\label{eq:UB2}
	\overline{u}_j^{n+1}= S_j^{UB}(\overline{u}^n):=\min\left(\overline{u}_j^{n+1}(\nu^m),\overline{u}_j^{n+1}(\nu^M)\right),~ {j\in \mathbb{Z}}
\end{equation}
For simplicity we will use the following short notation of \UB~scheme. We recall the simplified flux form of UB scheme which is used in~\cite{DL01} where
\begin{equation}\label{eq:ubadv}
u_{j+1/2}^{n}:=\overline{u}^{n}_j + \frac{1-\nu_j}{\phi_j} (\overline{u}^{n}_{j+1}-\overline{u}^{n}_{j}),
\end{equation}
where $\phi_j$ is defined as
\begin{equation}\label{eq:phi}
\phi_j = 
\left\{
\begin{array}{ll} 
\max \left ( 0, \min \left( \frac{2r_j}{\nu_j}, \frac{2}{1-\nu_j}\right) \right), &\text{ if}\quad  \overline{u}^{n}_{j+1}=\overline{u}^{n}_j \quad \text{and}\quad \nu_j\neq 1\\
0, & \text{otherwise},\\
\end{array}
\right.
\end{equation}
where $r_j =\frac{\overline{u}^{n}_j-\overline{u}^{n}_{j-1}}{\overline{u}^{n}_{j+1}-\overline{u}^{n}_{j}}$. Replacing  $j=j-1$ we can compute $u_{j-1/2}^{n}$.\\
Finally, as we said, the \UB~scheme has been proved to transport exactly step functions when the velocity is constant. We will see in \S \ref{sec:prop} that it can be written in  an incremental form and this will be useful in some proofs.
\subsection{Coupled scheme \cite{S2018}}
As previously mentioned, \UB~scheme build upon prior conservation law results and typically involve a discontinuous reconstruction at each step. While this choice proves effective in non-regular solution regions, it falls short in regular solution areas. Therefore, a promising approach is to blend the strengths of two schemes: one (\SL)~well-suited for regular (at least Lipschitz continuous) solutions and an \UB~scheme that excels in preserving solution profiles at jumps. By combining these two schemes, we anticipate obtaining several advantages. To achieve this, we need the ability to identify both regular and singular regions.
The SL scheme employs a local interpolation operator to recover the numerical solution's value at the points where characteristics intersect the grid, rather than using cell averages, as is the case with \UB~schemes.
For coupling we need two different grids $G^{SL}$ and $G^{UB}$ give in Fig. \ref{fig:grids}.
\begin{figure}
\centering
	\begin{subfigure}{0.5\textwidth}
		\includegraphics[width=\textwidth]{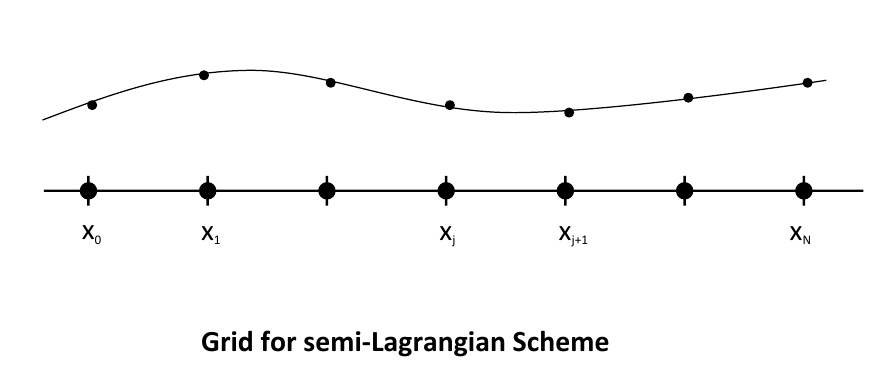}
		\caption{$G^{SL}:=\{x_j: x_{j}=j\Delta x, \; j=\mathbb{Z} \}$}
		\label{fig:first}
	\end{subfigure}
	\hfill
	\begin{subfigure}{0.5\textwidth}
		\includegraphics[width=\textwidth]{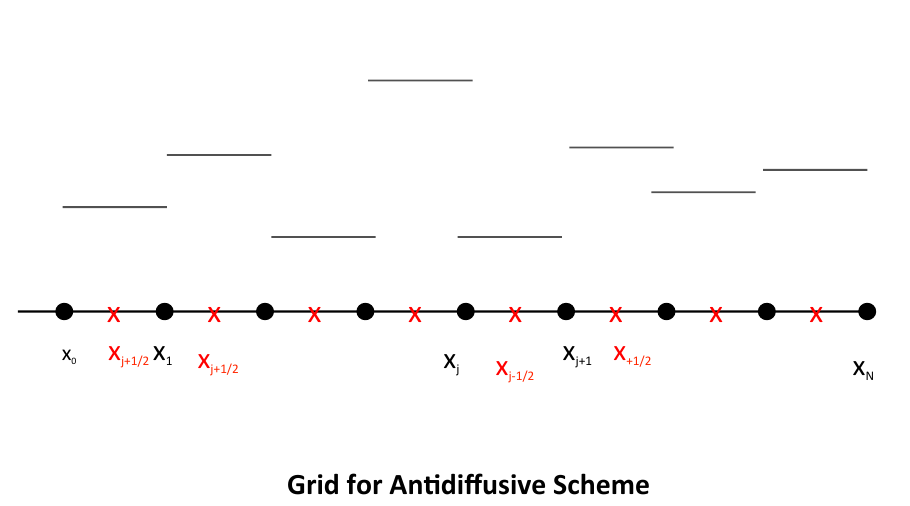}
		\caption{$G^{UB}:=\{\overline{x}_{j}: \overline{x}_{j}=x_{j}+\frac{\Delta x}{2},~j\in \mathbb{Z}\}$.}
		\label{fig:second}
	\end{subfigure}
		\caption{The two grids for \SL~and \UB~ schemes.}\label{fig:grid}
	\label{fig:grids}
\end{figure}
\bnodes the nodes of $G^{SL}$ and \cnodes  the nodes of $G^{UB}$.
In the sequel $u_{j}^{n}$ denotes an approximation of $u(x_{j},t_{n})$, and $\overline{u}_{j}^{n}$ denotes an approximation of $\overline{u}(\overline{x}_{j},t_{n})$, where $t_{n}=n\Delta t$, $\Delta t>0$. Moreover, we will drop the time index $n$ and denote for simplicity $u_{j} =u^{n}_{j}$ whenever the time dependence is not necessary.
At every step, we divide our domain into two regions, one where our approximate solution is ``regular'' and the other where we detect discontinuities. In order to make the readability and simplicity we recall all the definitions from \cite{S2018}.
Left and right derivatives for every node $x_j\in G^{SL}$  
\begin{equation}
	D^-u_{j}:=\frac{u_{j}-u_{j-1}}{\Delta x}~\hbox{ and }~
	D^+u_{j}:=\frac{u_{j+1}-u_{j}}{\Delta x} 
\end{equation}
\begin{definition}[Regular cell.]
Let $\delta$ be a positive threshold parameter. A cell $C_{j}=[x_{j}, x_{j+1})$ is said to be a {\em regular cell} if we have $|Du_{j}|< \delta$, $Du_{j}Du_{j-1}>0$ and $Du_{j}Du_{j+1}>0$. 
\end{definition}
\vspace*{0.25cm}
This means that a derivative below a given threshold as well as  a constant sign in the derivatives just before and after the node $x_j$ is considered to be a regularity indicator. For the choice of the threshold $\delta$ we can use a previous knowledge of the bounds for the exact solution. For example,  in the case of transport equation with the constant velocity, we know the solution which is $u(x,t)=u_{o}(x-ct)$, so we can set our threshold with the help of the initial condition $\delta=||Du^{0}_{j}||_{\infty}-\epsilon$, $\epsilon>0.$
\begin{figure}[!hbt]
\centering
\includegraphics[width=1\textwidth]{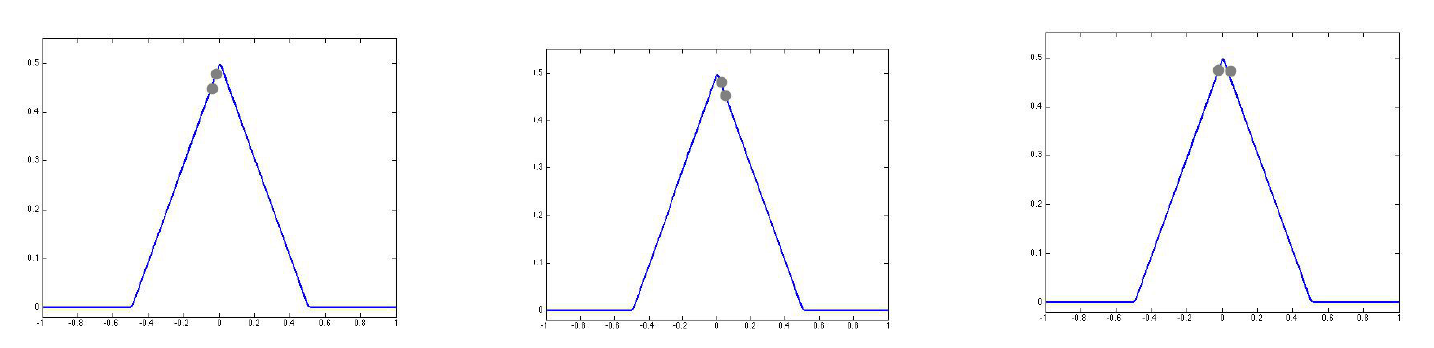}
\caption{ A sketch of three possible situations around a jump of the derivative.}
\end{figure}
\begin{definition}[Singular cell.]
A cell $C_{j}$ is said to be {\em a singular cell}  if it is not a regular cell. We denote the set of singular cells by $\mathcal{C}_{s}$.
\end{definition}
\vspace*{0.25cm}
\begin{definition}[Singular and regular region ]
The {\em singular region} $\Omega_{sin}$ is defined by the union of all the singular cells
The set $\Omega_{reg}=\mathbb{R}\setminus \Omega_{sin}$ is called the {\em regular region}.
\end{definition}
\vspace*{0.25cm}
\noindent
We need to distinguish between the nodes $x_j\in G^{SL}$ belonging to one of the above regions in order to apply the more adapt scheme there. To this end we define the {\em regularity indicator}, which will govern the switching between the two schemes: $\sigma_{j}\equiv 0$ for $x_j\in\Omega_{sin}$  and $\sigma_{j}\equiv 1$ for $x_j\in \Omega_{reg}$.  
\begin{definition} [{\em Local Projection Operator for SL}]
We define the local projection operator $P^{SL}: ~ \mathbb{R}^{2}\rightarrow \mathbb{R}$ by a map which defines the new value $u_j$ at $x_j$ starting from the values $(\overline{u}_{j-1/2},\overline{u}_{j+1/2})$,
\begin{equation}\label{def:locpsl}
P^{SL}( \overline{u}_{j-1/2}, \overline{u}_{j+1/2}):=\frac{\overline{u}_{j-1/2}+\overline{u}_{j+1/2}}{2}=u_{j}
\end{equation}
\end{definition}
\noindent
The $P^{SL}$ operator constructs the point value at $x_j$ as the average of the averaged values at $\overline x_{j-1}$ and $\overline x_j$.
\begin{definition} [{\em Local Projection Operator for UB}]
We define the local projection operator $P^{UB}: ~ \mathbb{R}^{2}\rightarrow \mathbb{R}$ by a map which defines the new value $\overline u_j$ at $x_{j+1/2}$ starting from the values $({u}_{j},{u}_{j+1})$,
\begin{equation}\label{def:locpad}
P^{UB}(u_j, u_{j+1}):=\frac{u_{j}+u_{j+1}}{2}=\overline{u}_{j+1/2}.
\end{equation}
\end{definition}
\noindent
The $P^{UB}$ operator constructs the averaged value at $\overline x_j$ as the average of the point values at $x_{j}$ and $x_{j+1}$.\\
The projection operators will be used locally whenever in a cell we switch from one scheme to the other and we need new values which were not available before. The $P^{SL}$ operator will also be used at Step 5 to allow the up-date of the regularity indicator which is computed on the \bnodes. In the sequel we will consider an initial condition $w^0$ with compact support $Q$ and define the subset $J:=[j_{min}, j_{max}]\subset\mathbb{Z}$ containing the node indices of an interval containing $Q$. Now we will give the algorithm for coupled shceme \ref{CSalgo} from \cite{S2018}
\begin{algorithm}[H]
	\caption{Algorithm for the Coupled (\SL+\UB) scheme}
	\noindent\makebox[\linewidth]{\rule{\textwidth}{0.4pt}}
	{\bf{Initialisation:}} We compute the initial data ${w_{j}^{0}=u^0_j }$ on every $x_j$, $j\in J$.\\
	We  compute $D^-w^0_{j-1}$, $D^-w^0_{j}$ and $D^-w^0_{j+1}$  and check the condition 
	\begin{equation} \label{cond:0}
		|D^-w^0_{j}|< \delta \hbox{ and }D^-w^0_{j-1}D^-w^0_{j}>0,~~ D^-w^0_{j}D^-w^0_{j+1}>0  \;.
	\end{equation}
	if condition \eqref{cond:0} is true then we set $\sigma^0_j=1$ else  $\sigma^0_j=0$.\\
	
	{\bf Main cycle  on $j\in J$},  For $n>0$. 
	\begin{algorithmic}
		\State {\bf Step 1.}  We  compute $D^-w^n_{j-1}$, $D^-w^n_{j}$ and $D^-w^n_{j+1}$  and check the condition 
		\begin{equation} \label{cond:1}
			|D^-w^n_{j}|< \delta \hbox{ and }D^-w^n_{j-1}D^-w^n_{j}>0,~~ D^-w^n_{j}D^-w^n_{j+1}>0  \;.
		\end{equation}
	    If condition \eqref{cond:1} is true then go to Step 2 else we go to Step 3.
	    \State{\bf Step 2.} We apply the \SL~scheme at $x_j$ and set $\sigma^n_{j}= 1$ at the node $x_j$.\\ 
	    If $\sigma_j^n=\sigma_j^{n-1}$ we directly compute the new value
	    according to the\SL~cheme
	    \begin{equation}
	    	w_{j}^{n+1}= \sigma^n_{j}S_j^{SL}[w^{n}]+(1-\sigma^n_{j})S_j^{UB}[w^{n}]=S_j^{SL}[w^n].
	    \end{equation} 
	    If $\sigma_j^n\not=\sigma_j^{n-1}$, we have to switch from  the $AD$-scheme to the \SL~scheme and we need the projection $P^{SL}$. Then, we set for $k=j, j+1$
	    $$w^n_k=P^{SL}( \overline{u}^n_{k-1/2}, \overline{u}^n_{k+1/2}):=\frac{\overline{u}^n_{k-1/2}+\overline{u}^n_{k+1/2}}{2}=u^n_{k}$$
	    and we compute 
	    \begin{equation}
	    	w_{j}^{n+1}= \sigma^n_{j}S_j^{SL}[w^{n}]+(1-\sigma^n_{j})S_j^{UB}[w^{n}]=S_j^{SL}[w^n],
	    \end{equation} 
%
If $\sigma_j^n\not=\sigma_j^{n-1}$, we have to switch from  the $AD$-scheme to the \SL~scheme and we need the projection $P^{SL}$. Then, we set for $k=j, j+1$
$$w^n_k=P^{SL}( \overline{u}^n_{k-1/2}, \overline{u}^n_{k+1/2}):=\frac{\overline{u}^n_{k-1/2}+\overline{u}^n_{k+1/2}}{2}=u^n_{k}$$
and we compute 
\begin{equation}
	w_{j}^{n+1}= \sigma^n_{j}S_j^{SL}[w^{n}]+(1-\sigma^n_{j})S_j^{UB}[w^{n}]=S_j^{SL}[w^n],
\end{equation} 

\end{algorithmic}\label{CSalgo}
\end{algorithm}
\begin{algorithm*}
\begin{algorithmic}	
	\State {\bf Step 4.} The condition \eqref{cond:1} is not satisfied, then  we set $\sigma^n_{j}= 0$.\\
	If $\sigma_j^n=\sigma_j^{n-1}$ we directly compute the new value
	according to the $AD$-scheme
	\begin{equation}
		\overline w_{j}^{n+1}= \sigma^n_{j}S_j^{SL}[w^{n}]+(1-\sigma^n_{j})S_j^{UB}[w^{n}]=S_j^{UB}[w^n].
	\end{equation} 
	If $\sigma_j^n\not=\sigma_j^{n-1}$, we have to switch from  the $SL$-scheme to the $AD$-scheme and we need the projection $P^{UB}$. Then, we set for $k=j-1, j, j+1$
	$$\overline w^n_k=P^{UB}( {u}^n_{k}, {u}^n_{k+1}):=\frac{{u}^n_{k}+{u}^n_{k+1}}{2}=\overline u^n_{k+1/2}$$
	and we compute 
	\begin{equation}
		\overline w_{j}^{n+1}= \sigma^n_{j}S_j^{SL}[w^{n}]+(1-\sigma^n_{j})S_j^{UB}[\overline w^{n}]=S_j^{UB}[\overline w^n],
	\end{equation} 
	{\bf End of the $j$ cycle}.
\State
{\bf Step 5} (Filling the holes procedure)\\
	At  the \bnodes where $\sigma_j^n=0$ we need to project by $P^{SL}$ defined in \eqref{def:locpsl} using the intermediate  values at $\overline w_{j-1}^{n+1}$ and $\overline w_{j}^{n+1}$, i.e.
	$$w^{n+1}_j=P^{SL}(\overline w_{j-1}^{n+1}, \overline w_{j}^{n+1}), \hbox{ for } \sigma_j=0 .$$
	(the values $w_{j}^{n+1}$ for the \bnodes  where $\sigma_j=1$ are already available by Step 4). This will finally produce  the new approximate solution $w^{n+1}_j$.
\State	
{\bf Step 6}
Set $n=n+1$, $j=j_{min}$ and go back to the Main cycle. \hfill$\Box$
\end{algorithmic}	
\vspace{0.3cm}
Note that at the \bnodes where $\sigma^n_j=1$ we always have a value which is computed by the \SL~scheme and that the switching indicator is chosen on the basis of the values at the \bnodes.
\end{algorithm*}
\section{Some properties of the coupled \SL+\UB~ scheme}\label{sec:prop}
This section is devoted to the analysis of some interesting  properties for  the coupled scheme. We will study these  properties for the advection problem, the extension to the non linear problem is rather difficult and  is still under study. However, at the end of the paper we will present also a test for an HJ equation which shows that the coupling procedure is also effective for nonlinear problems and deserves further analysis. Let us start introducing some classical definitions.
\begin{definition}\label{def:TVD}
The {\em  Discrete Total Variation} of a vector $u=\{u_j\}_{j\in\mathbb{Z}}$ is given by
\begin{equation}
TV(u):= \sum_{j\in \mathbb{Z}}|u^n_{j}-u^n_{j+1}|. 
\end{equation}
\end{definition} 
\noindent
This definition is the discrete analogue of the continuous total variation for a continuous function.  
\begin{definition} 
We say that a scheme is {\em Total Variation Diminishing (TVD)}  if for all $n\geq 0$ , 
\begin{equation}
TV(u^{n+1}) \leq TV(u^n).
\end{equation}
where $u^n$ is the approximate solution at time $t_n$.
\end{definition}
\begin{definition} 
We say that a scheme is {\em Total Variation Bounded (TVB)} if for any initial condition such that $TV(u^0)<\infty$ and time $T$ there exists a positive constant $C$, and a value $\Delta t_0$ such that
\begin{equation}
TV(u^{n}) \leq C
\end{equation}
for all  $n\Delta t\le T$ whenever $\Delta t<\Delta t_0$ (again $u^n$ denotes the approximate solution at time $t_n$).
\end{definition} 
This definitions dates back to the first papers on high-order approximation schemes for scalar conservation laws (see \cite{H83, H84}). The fact that the total variation is decreasing in time is a typical feature of entropy solution to scalar conservation laws. Moreover, the control of the total variation gives a control on the oscillations of the scheme.  For a detailed analysis of the role of the above property in the analysis of high-order approximation schemes we refer to the monographs  \cite{L92} and  \cite{GR96}.
Following Harten \cite{H84} say that a {\em scheme is in incremental form} if it can be written as 
\begin{equation}\label{eq:TVDH}
u_j^{n+1}= u_{j}^n-C_{j-\frac{1}{2}} (u^n_j-u^n_{j-1})
+D_{j+\frac{1}{2}} (u^n_{j+1}-u^n_{j}),
\end{equation}
where $C_{j-\frac{1}{2}}, ~D_{j+\frac{1}{2}} \in \mathbb{R}$.  We recall  \cite {H83} that a scheme in incremental form is TVD if and only if the following sufficient  conditions are satisfied for all $j$:
\begin{equation}
0 \leq C_{j-\frac{1}{2}},~D_{j+\frac{1}{2}}~\text{and}~C_{j-\frac{1}{2}}+D_{j+\frac{1}{2}}\leq 1.
\end{equation}
For the \UB~scheme it is relevant to recall the following definitions.
\begin{definition}\label{def:stableinf} The \UB~scheme is {\em $L^\infty$- stable} if the following conditions hold:
\begin{eqnarray}\label{eq:stableinf}
\hbox{for }\nu_j \geq 0, \; \min({u}^n_{j},{u}^n_{j-1}) \leq{u}^{n+1}_{j} \leq \max({u}^n_{j},{u}^n_{j-1}),\\
\hbox{for }\nu_j < 0, \;  \min({u}^n_{j},{u}^n_{j+1}) \leq{u}^{n+1}_{j} \leq \max({u}^n_{j},{u}^n_{j+1}).
\end{eqnarray}
\end{definition} 
\noindent
It is clear that above definition of $L^\infty$- stability implies the standard definition of $L^\infty$- stability  which requires
\begin{equation}\label{eq:SL_inf_sta}
\|{u}^{n}\|_{L^\infty} \leq C  \|{u}^0\|_{L^\infty}, \forall n\in\mathbb{N}
\end{equation}
Infact,  by definiton \ref{def:stableinf} one gets 
\begin{equation}
-\Vert u^n\Vert _\infty\le  \min(u^n_{j-1}{u}^n_{j},{u}^n_{j+1}) \leq{u}^{n+1}_{j} \leq \max(u^n_{j-1}, {u}^n_{j},{u}^n_{j+1})\le\Vert u^n\Vert _\infty
\end{equation}
which easily implies \eqref{eq:SL_inf_sta}.
\begin{definition}\label{def:cons}
We say that UB shceme is {\em consistent} if all the fluxes ${u}^{n, L}_{j+\frac{1}{2}}$ and ${u}^{n, R}_{j+\frac{1}{2}}$ satisfy:
\begin{eqnarray}\label{eq:consist}
\hbox{for }\nu_j \ge  0, \;\min(\overline{u}^n_{j},\overline {u}^{n}_{j-1} ) \leq {u}^{n,L}_{j+\frac{1}{2}}\leq \max(\overline{u}^n_{j},\overline{u}^n_{j-1}),\\
\hbox{for }\nu_j  < 0, \;  \min(\overline{u}^n_{j},\overline{u}^{n}_{j+1} ) \leq {u}^{n,R}_{j+\frac{1}{2}}\leq \max(\overline{u}^n_{j},\overline{u}^n_{j+1}).
\end{eqnarray}
\end{definition} 
\noindent
As we said, if  $|\nu_j|\le 1$ for every $j$,  the \UB~ scheme is consistent, $L^\infty$ stable and TVD. These properties will now be extended to the coupled scheme using the definition of our projection operators on the two grids $G^{SL}$ and $G^{UB}$.

\vspace{0.3cm}\noindent
{\bf Properties of the Coupled Scheme for the advection equation}\\
Let us consider the following model problem
\begin{equation}\label{eq:adv}
\left\{ \begin{array}{l}
u_t+cu_x=0, ~ x \in \mathbb{R}, \; t\in [0,T]  \\
u(x,t)=u_0(x)
\end{array} \right.\end{equation}
where $c$ is a constant velocity. In the following, we will continue to use the notations $u^n_j$, $\overline{u}^n_j$ and $w^n_j$ respectively for the values computed by the  \SL,~\UB~ and coupled scheme at time $n$ and at the node $j$ of their respective grids (shifted by $\Delta x/2$). 
We consider the particular coupled scheme obtained by a the \SL~scheme and the \UB~scheme:
\begin{equation}
w_{j}^{n+1}= \sigma^n_{j}S_j^{SL}[w^{n}]+(1-\sigma^n_{j})S_j^{UB}[w^{n}],
\end{equation}
with the two projections \eqref{def:locpsl} or \eqref{def:locpad} (as explained in the coupled scheme algorithm we use projection only when it is needed by $\sigma^n_j$).
When we are in the regular region  the above coupled scheme coincides with the \SL~scheme. For $c>0$, let $\nu:= c\frac{\Dt}{\Dx}<1$, we have   $x_j-c\Dt \in (x_{j-1}, x_j]$ obtaining the following \SL~scheme 
\begin{equation}\label{eq:sladv}
w_{j}^{n+1}=u^{n+1}_j= S_j^{SL}(u^{n}) :=\nu u^n_{j-1} + (1-\nu) u^n_{j}
\end{equation}
Although $\Dt$ can in general be rather big  as \SL~schemes typically work for large Courant numbers here we will set $\nu=1$ because the coupling is made with the \UB~scheme which needs that condition for stability. This limitation will be compensated by the higher accuracy at the jumps given by the \UB~scheme.
Note that the following properties of the projection operators play an important role:
\begin{equation}\label{eq:4}
\min({u}^n_{j} ,{u}^n_{j+1} ) \leq \overline{u}^n_{j} = P^{UB}  (u^n_{j}, u^n_{j+1} )=\frac{u^n_{j}+u^n_{j+1}}{2}\leq \max ({u}^n_{j} ,{u}^n_{j+1} )
\end{equation}
\begin{equation}\label{eq:5}
\min({\overline u}^n_{j} ,{\overline u}^n_{j+1} ) \leq u^n_{j} = P^{SL}  (\overline  u^n_{j}, \overline u^n_{j+1} )=\frac{\overline u^n_{j}+ \overline u^n_{j+1}}{2}\leq \max ({\overline u}^n_{j} ,{\overline u}^n_{j+1} )
\end{equation}
\noindent
In order to clarify which values are really involved in the computation we will keep the notations with $u$ and $\overline u$ instead of $w$. However, these values are computed according to the coupled algorithm already described in \S 3.
\begin{proposition}\label{prop1}
Let us consider the advection problem \eqref{eq:adv} and let $|\nu| \le 1$. The coupled scheme SL+UB is $L^\infty$-stable.
\end{proposition}
\begin{proof}
Note that for a constant velocity $c$, choosing  $\lambda=\Delta t/\Delta x$ small enough the parameter $\nu$ will be always lower than 1. This means that only the first neighbouring cells will appear in the stencil for the \SL~and for the \UB~scheme. 
We will consider four cases at a generic node $x_j$: two are related to the situation where the parameter $\sigma^n_j$ remains constant passing from step $t_{n-1}$ to $t_n$ whereas the remaining two cases refer to the switching case.\\

\noindent
{\bf Case 1:} $\sigma^n_j=\sigma^{n-1}_j= 1$, i.e.~no switch is needed. We continue to apply at the node $x_j$ the \SL~scheme with local piecewise linear reconstruction on the neighbouring cells $[x_{j-1}, x_j]$ and $[x_j, x_{j+1}] $ so 
$$ \min({u}^n_{j-1} ,  {u}^n_{j}, {u}^n_{j+1} ) \leq {u}^{n+1}_{j}\le  \max({u}^n_{j-1} ,  {u}^n_{j}, {u}^n_{j+1} ).$$
and this implies the same for $w$.\\

\noindent
{\bf Case 2:} $\sigma^n_j=\sigma^{n-1}_j= 0$, i.e. no switch is needed. We continue to apply the \UB~scheme and the property at the node $j$ comes from the fact that the \UB~scheme satisfies the stability property  \eqref{def:stableinf}.\\

\noindent
{\bf Case 3:} $\sigma^n_j=1$ and $\sigma^{n-1}_j= 0$. We switch from the \UB~scheme to the \SL~scheme so we need to use local projection operator \eqref{def:locpsl}.
More precisely, we assign (if necessary) to the \bnodes  $x_{j-1}$, $x_j$ and $x_{j+1}$  the values obtained by averaging the values computed by the \UB~scheme at the corresponding \cnodes and we apply the \SL~scheme to compute $u^{n+1}_j$. This will produce a new value satisfying
$$  \min(\overline{u}^n_{j}, \overline{u}^n_{j+1} ) \leq {u}^{n+1}_{j}\le  \max( \overline{u}^n_{j}, \overline{u}^n_{j+1} )$$
which by construction implies
$$ \min({u}^n_{j-1} ,  {u}^n_{j}, {u}^n_{j+1}) \leq {u}^{n+1}_{j}\le  \max({u}^n_{j-1} ,  {u}^n_{j}, {u}^n_{j+1} ).$$

\noindent
{\bf Case 4:} $\sigma^n_j=0$ and $\sigma^{n-1}_j= 1$. We switch from the \SL~scheme to the \UB~scheme so we need to use the local projection operator \eqref{def:locpad}. 
More precisely, the values $u^n_{j}$ has been computed by \SL~scheme because $\sigma_{j}^{n-1}$=1 and we assume that also the neighbouring values $u^n_{j-1}$ and $u^n_{j+1}$ are available (if not they can be obtained averaging by the projection operator $P^{UB}$ \eqref{def:locpad}). \\
This gives   
\begin{equation}\label{eq:7}
P^{UB}(w^{n}_{j-1}, w^n_{j})=P^{UB}(u^n_j, u^n_{j+1})=\frac{u^n_{j}+u^n_{j+1}}{2}=\overline{u}^n_{j},
\end{equation}
and since $\sigma_{j}^{n}=0$ the coupled scheme will compute the  value
\begin{equation}\label{eq:8}
w_{j}^{n+1}=S_j^{UB}[w^{n}]
\end{equation}
Then the $L^\infty$  bound  is satisfied by \eqref{eq:4} and the stability property of the  \UB~scheme.
\end{proof}

\begin{proposition}\label{prop2}
We consider the advection problem \eqref{eq:adv} and let $| \nu| \le 1$. The coupled scheme SL+UB is TVB. 
\end{proposition}
\begin{proof}
To prove the TVB property we still have to examine four cases as in the previous proposition. We will give the proof for $\nu\ge0$, when $\nu$ is negative the proof can be easily adapted. As we will see, in some cases when we do not switch we will have a stronger property, i.e. the scheme will be TVD. When we have a switch we just have the TVB property.

\vspace{0.2cm}\noindent
{\bf Case 1:} For $\sigma^n_j=\sigma^{n-1}_j=1$ for every $j$, so no switch is needed and we will always apply \SL~at all the nodes. 
Let us prove that the scheme is TVD. We have 
\begin{equation}\label{eq:9}
u^{n+1}_{j}=w^{n+1}_{j}= S_{j}^{SL}[w^{n}]
\end{equation}
which means 
\begin{equation}\label{eq:sl_p1}
u^{n+1}_j= \nu u^n_{j-1} + (1-\nu) u^n_{j}
\end{equation}
where $\nu \in (0,1]$. So for the difference we get
\begin{eqnarray}
|u^{n+1}_{j+1}-u^{n+1}_j|&=& | \nu u^{n}_{j} + (1-\nu) u^{n}_{j+1}-(\nu u^{n}_{j-1} + (1-\nu) u^{n}_{j})| \le\\
&\le&  | \nu (u^{n}_{j} -u^{n}_{j-1} )+(1-\nu) (u^{n}_{j+1} -u^{n}_{j} )|\nonumber
\end{eqnarray}
Summing on $j$ we obtain
\begin{equation}
\sum_{j\in \mathbb{Z}}|u^{n+1}_{j+1}-u^{n+1}_j|\leq \sum_{j\in \mathbb{Z}} | \nu (u^{n}_{j} -u^{n}_{j-1} )+(1-\nu) (u^{n}_{j+1} -u^{n}_{j} )|
\end{equation}
so
\begin{equation}
TV(u^{n+1}) \leq \nu TV(u^{n})+ (1-\nu) TV(u^{n})
\end{equation}
for $ \nu \in [0,1]$, which implies for the approximate solution of the coupled scheme 
\begin{equation}\label{eq:sl_tvd}
TV(w^{n+1}) =TV(u^{n+1}) \leq TV(u^{n})= TV(w^{n}).
\end{equation}

\vspace{0.2cm}\noindent
{\bf Case 2:} For $\sigma^n_j=\sigma^{n-1}_j=0$ for every $j$, so no switch is needed and we will always apply \UB~at all the nodes.  For $\nu\ge0$ we have 
\begin{equation}\label{eq:10}
w_{j}^{n+1/2}=S_j^{UB}[w^{n}]
\end{equation}
so we can also write for every $j\in \mathbb{Z}$
$${w}^{n+1}_j= \overline{u}^{n+1}_j= \overline{u}^n_j- C_{j-1/2}(\overline{u}_j^n-\overline{u}_{j-1}^n)$$
with $C_{j-1/2}\in [0,1]$, i.e.~\eqref{eq:TVDH} with $D_{j+ 1/2}=0$. 
Hence we have the incremental form \eqref{eq:TVDH} with $C_{j+ 1/2}+D_{j+1/2} \leq 1$. 
Thus the scheme is TVD. For $\nu<0$ we will have the a similar expression where the coefficient $C_{j+ 1/2}$ vanishes and $D_{j+ 1/2}>0$, so again we will have the TVD property.\\

\vspace{0.2cm}\noindent
{\bf Case 3:}  $\sigma^n_j=1$ and $\sigma^{n-1}_j= 0$. In addition, we assume that $\sigma^n_{j-1}=1$. The scheme switches at $x_j$ from the \UB~to \SL.~For $\nu\ge 0$, let us examine the total variation in the interval $D=[x_{j-1}, x_{j+1}]$, i.e. in the union of cells whose nodes are used in the switch. We define ${J}:=\{j-1,j\}$ and we denote by   $Var_{{ A}}(w)$ the variation of a vector $w$ over ${ A}$, i.e.

\begin{equation}\label{eq:varj}
Var_{{J}}(w):= \sum_{k=j-1}^j |w_k-w_{k+1}|
\end{equation}
Recalling the definition of $P^{SL}$ we have 
\begin{eqnarray}\label{eq:var1}
Var_J(w^{ {n+1}})&&= \sum_{k=j-1}^j |w_k^{ {n+1}}- w_{k+1}^{ {n+1}}|\\
&& =|w_{j-1}^{ {n+1}}- w_j^{ {n+1}}|+|w_j^{ {n+1}}-w_{j+1}^{ {n+1}}|\\
&& =\left |u{ ^{n+1}_{j-1}} - \frac{\overline{u}_{j-1}^{ ^{n+1}}+\overline{u}_{j}^{ ^{n+1}}}{2}\right |+ \left |\frac{\overline{u}_{j-1}^{ ^{n+1}}+\overline{u}_{j}^{ ^{n+1}}}{2}-u_{j+1}^{ ^{n+1}} \right |\nonumber
\end{eqnarray}
If there is a switch from \SL~to \UB~scheme then by equation \eqref{cond:0} and \eqref{cond:1}, we have
\[
|D^-w_j|>\delta.
\]
If we are in regular region that means $\sigma^n_j=1$  then we have
\[
\big|\frac{w^n_j-w^n_{j-1}}{\Delta x} \big|< \delta \\ 
\Rightarrow
|w^n_j-w^n_{j-1}|< \delta \Delta x
\]
\noindent
Now we apply the \SL~scheme to these nodes and we assume ${ \nu}\ge 0$ (the opposite sign can be treated in a similar way).    For every $k\in A$ we have 
\begin{equation}\label{eq:SLevol}
w^{n+1}_{k}= S_{k}^{SL}[w^{n}]= \nu w^n_{k-1}+(1-\nu)w^n_{k}
\end{equation}
Now we want to obtain a bound for $Var_J(w^{n+1})$. By applying \eqref{eq:SLevol} and simply reordering the terms as we have done in the above proof of Case 1, we have
\begin{equation}
Var_{ A}(w^{n+1})=\sum_{k=j-1}^j |\nu (w^n_{k-1}-w^n_k)+(1-\nu)(w^n_k-w^n_{k+1})|
\end{equation}
{  which implies, since $\nu$ and $(1-\nu) \in[0,1]$  ,
\begin{eqnarray}\label{eq:boundJ}
\quad\quad Var_{A}(w^{n+1})&&\leq |w^n_{j-2}-w^n_{j-1}|+ |w^n_{j-1}-w^n_{j}|+ |w^n_{j}-w^n_{j+1}|\\\nonumber
&&\leq |w^n_{j-2}-w^n_{j-1}|\}+ Var_J(w^n)\\\nonumber
&&\leq Var_{J \cup \{j-2\}}(w^n). \nonumber
\end{eqnarray}}
which implies, since $\nu\in[0,1]$ 
\begin{eqnarray}\label{eq:boundJ1}
\quad\quad Var_J(w^{n+1})&&=\nu |w^n_{j-2}-w^n_{j-1}|+|w^n_{j-1}-w^n_j|+(1-\nu)|w^n_j-w^n_{j+1}|\\
&&\le 2 \nu \max\{|w^n_{j-2}|,|w^n_{j-1}|\}+ Var_J(w^n) \nonumber
\end{eqnarray}
To obtain a uniform bound for every time horizon let us take $T=N \Delta T$ and denote by $\overline M$ the maximum number of switches at every iteration, clearly $\overline M$ is bounded by the total number of nodes $M$  in the (compact) support of the solution $w^n$ (we can always assume that they are all contained in the interval $[a,b]$ and that $\Delta x=(b-a)/M$). Let us also denote by $J^n_0$  the set of indices corresponding to the nodes where at time $t_n$ there is no switch and by $J^n_1$ the set of nodes where there is a switch. Clearly, at time $t_n$ a node must belong either to $J^n_0$ or to $J^n_1$. {  Let us consider number of elements is  $J_1$ is $\bar{M}_n$}. For every $0<n\le N-1$,  we have
{ 
\begin{eqnarray}
\quad TV&&(w^{n+1})= \sum_{k\in {J^n_0}} |w^{n+1}_k-w^{n+1}_{k+1}|+\sum_{k\in {J^n_1}} |w^{n+1}_k-w^{n+1}_{k+1}|+ \bar{M_n} \delta \Delta x\\\nonumber
&&\le\sum_{k\in {J^n_0}\cup {J^n_1}} |w^{n}_k-w^{n}_{k+1}|+  \bar{M}_n\delta \Delta x\\\nonumber
&&\le\sum_{k\in {J^n_0}\cup {J^n_1}} |w^{n}_k-w^{n}_{k+1}|+  M \delta \Delta x, \qquad \text{since}\quad\bar{M}_n \leq M\\\nonumber 
&&\le TV(w^n)+ (b-a)\delta.
\end{eqnarray}
Other possibility
\begin{eqnarray}
\quad TV&&(w^{n+1})= \sum_{k\in {J^n_0}} |w^{n+1}_k-w^{n+1}_{k+1}|+\sum_{k\in {J^n_1}} |w^{n+1}_k-w^{n+1}_{k+1}|\\\nonumber
&&\le 2\sum_{k\in {J^n_0}} |w^{n}_k-w^{n}_{k+1}|+   \sum_{k\in {J^n_1}} |w^{n}_k-w^{n}_{k+1}|\\\nonumber
&&\le TV(w^n)+\sum_{k\in {J^n_0}} |w^{n}_k-w^{n}_{k+1}|\\\nonumber 
&&\le TV(w^n)+ \delta \Delta x Card(J^n_0)\\\nonumber 
&&\le TV(w^n)+ \delta \Delta x M\\\nonumber
&&\le TV(w^n)+ (b-a)\delta. 
\end{eqnarray}
In both the cases we obtained the same bound.

}

\begin{eqnarray}
\quad TV&&(w^{n+1})= \sum_{k\in {J^n_0}} |w^{n+1}_k-w^{n+1}_{k+1}|+\sum_{k\in {J^n_1}} |w^{n+1}_k-w^{n+1}_{k+1}| \\
&&\le\sum_{k\in {J^n_0}} |w^{n+1}_k-w^{n+1}_{k+1}|+\sum_{k\in {J^n_1}} |w^{n}_k-w^{n}_{k+1}|+2\overline{M} \nu \Vert w^n\Vert_\infty\nonumber\\
&&\le\sum_{k\in {J^n_0}\cup {J^n_1}} |w^{n}_k-w^{n}_{k+1}|+ 2\overline{M} \nu \Vert w^n\Vert_\infty= TV(w^{n})+ 2\overline{M} \nu \Vert w^n\Vert_\infty\nonumber
\end{eqnarray}
where we have used the fact that for $j\in J^n_0$ the total variation is non increasing. By the $L^\infty$ bound proved in Proposition \ref{prop1} we obtain $\Vert w^n\Vert_\infty\le \Vert w^0\Vert_\infty$. Recalling the definition of $\nu$, we iterate back to $n=0$ obtaining
\begin{eqnarray}
\quad\quad TV&&(w^{n+1})\le TV(w^{n})+ 2\overline{M} \nu \Vert w^n\Vert_\infty\le TV(w^0)+ 2\overline{M}  n\nu \Vert w^0\Vert_\infty\\
&&\le TV(w^0)+ 2M^2 c \frac{T}{b-a} \Vert w^0\Vert_\infty\nonumber
\end{eqnarray}
and this gives the uniform bound for the total variation.
The proof for $\nu<0$ can be easily adapted.

\vspace{0.2cm}\noindent
{\bf Case 4:} $\sigma^n_j=0$ and $\sigma^{n-1}_j= 1$. In addition, we assume that $\sigma^n_{j-1}=0$. The scheme switches from the \SL~to \UB~scheme. 
We can first get a bound for the variation on the cell next to the $x_j$ node by applying the projection $P^{UB}$. Then, we can divide the indices into two subsets as in Case 3, and we can obtain a similar upper bound on $TV(w^{n+1})$ by using the $L^\infty$ bound for the \UB~scheme. The proof follows in the same way as in Case 3.
\end{proof}

\begin{proposition}\label{prop3}
Let $ |\nu| \le 1$, then the coupled scheme \SL+\UB  is consistent with equation \eqref{eq:adv}. 
\end{proposition}
\begin{proof}
Consistency is a local property and  it will be inherited by the same property of the two schemes used to construct the coupled scheme because, as we will see below, the two projection operators are defined as centred averages.

\vspace{0.2cm}\noindent
{\bf Case 1:} $\sigma^n_j=\sigma^{n-1}_j=1$  no switch is needed. We already know that \SL~scheme is consistent and that  is locally first order accurate (see  \cite{FF14} for details), so the property is true.

\vspace{0.2cm}\noindent
{\bf Case 2:} $\sigma ^n_j=\sigma^{n-1}_j=0$ and no switch is needed. \UB~scheme is also consistent according to \eqref{def:cons} (see \cite{BFZ10} for details).\\

\vspace{0.2cm}\noindent
{\bf Case 3 :} $\sigma ^n_j=1$ and $\sigma^{n-1}_j=0$ so we switch from the \UB~scheme to the \SL~scheme. The property follows from the way we have defined the projection  $P^{SL}$ on the grid and the consistency of the \SL~scheme.
We need to project on $x_j$, using (at most) the values at the neighbouring \cnodes computed at the previous iteration (the index $n$ is dropped for simplicity)
\begin{equation*}
u(x_j)=\frac{\overline{u}_{j-1}+\overline{u}_j}{2}
\end{equation*}
Then, by construction, 
\begin{equation*}
u(x_j)= \frac{1}{2\Delta x} \int_{x_{j-1}}^{x_{j}} u(x) dx + \frac{1}{2\Delta x} \int_{x_{j}}^{x_{j+1}} u(x) dx = \frac{1}{2\Delta x} \int_{x_{j-1}}^{x_{j+1}} u(x) dx 
\end{equation*}
Since $u(x_j) 2\Delta x$ is the approximate value corresponding to the mid-point rule applied to the integral in $[x_{j-1}, x_{j+1}] $, for a regular function we get the following estimate
\begin{equation}
|u(x_j)-  \frac{1}{2\Delta x} \int_{x_{j-1}}^{x_{j+1}} u(x) dx |\le C \Delta x^2
\end{equation}
which, by the consistency of the \SL~scheme, guarantees the local consistency for the coupled scheme.

\vspace{0.2cm}\noindent
{\bf Case 4 :} $\sigma ^n_j=0$ and $\sigma^{n-1}_j=1$ so we switch from the \SL~scheme to the \UB~scheme.
Now we need (at most)  the values at the neighbouring \cnodes with respect to $x_j$ and they can obtained by projection. The projection $P^{UB}$ defines (again the index $n$ is dropped for simplicity)
\begin{equation*}
\overline u_j=\frac{{u}_{j+1}+{u}_j}{2}
\end{equation*}
so recalling that  $\overline x_j=x_j+\Delta x /2$, we have 
\begin{equation}
\min (u_{j}, u_{j+1})\le \overline{u}_j\le \max ( u_{j}, u_{j+1})
\end{equation}
For the definition of the fluxes of the \UB~scheme we will also need $\overline{u}_{j-1}$ or $\overline{u}_{j+1}$ but also for these values we will have  similar bounds.
This implies that the coupled scheme is consistent according to the definition \eqref{def:cons}.
\end{proof}


\section{Numerical tests}\label{sec:cs_numtests}
In this \S, we present some numerical tests in one-dimension. We use different initial conditions with varying smoothness and track their time evolution over $\Omega$. The coupled scheme improves the accuracy by switching between schemes using $\sigma_j^n$. The extra cost of computing $\sigma_j^n$ is minimal as we only project the cells that switch. We start solving the advection equation with constant and variable velocity. Then, we give an example for an evolutive HJ equation~\eqref{eq:hj1} where, starting from a smooth initial condition, we follow the onset of a singularity at an intermediate time. We compare the proposed coupled scheme with the two schemes used as building blocks. To this end, we will consider several initial conditions with various regularity properties and we follow their evolutions in time over an interval $\Omega$. We will compute the errors in $L^1(\Omega)$, $L^2(\Omega)$ and, in some cases, in $L^\infty (\Omega_{reg})$ to show also the behavior in the regular region.  In our examples $\Omega_{reg}:=\Omega\setminus \Omega_{sing}$ and $\Omega_{sing}:= \cup_s B(\overline x_s, \varepsilon)$ where $\overline x_s$ denotes a point where the derivative or the solution itself has a jump.\\
\begin{example}\label{ex:adv} 
{\bf Advection equation with constant velocity.}\\
We consider the  advection equation \eqref{eq:adv}
where $c \equiv1$ is the velocity and $v_0(x)$ is the initial condition with bounded support, $\Omega:=(-2,2)$, $T=2$ and the Courant number $\nu=c{\Delta t}/{\Delta  x}$ which remains constant at 0.9 throughout all the simulations. In this example, we explore two different initial conditions, referred to as \eqref{eq:smooth} and \eqref{eq:jump}. \\
\noindent
{\em Test 1: Smooth initial condition $v_0$.} 
\begin{equation}\label{eq:smooth}
v_0(x)=\left\{\begin{array}{ll}
(1-|x|^2)^4, &\quad \hbox{for }  |x|\leq1\\
0&\quad\textrm{otherwise}
\end{array}\right.
\end{equation}
note that the derivative is 0 at the junction points $x=\pm 1$. It is clear that the solution remains smooth in the evolution hence the SL~scheme should have a better accuracy with respect to the UB~scheme. Moreover, since the slope is not high, we expect the coupled scheme to select always the SL~scheme. Fig.~\ref{fig:smooth_adv1}, shows the solution of~\eqref{eq:adv} at time $t=20\Delta t$ with time step $\Delta t=0.045$ for the initial data~\eqref{eq:smooth}. 
Fig. ~\ref{fig:sigma_smooth_adv1} shows the plots of  the switching parameter $\sigma$ for different time $t=10\Delta t,~20\Delta t,~30\Delta t$, with $\Delta t=0.045$.  As we expect,  switching parameter $\sigma\equiv 1$ for the coupled scheme. Table~\ref{tab:adv_ub_smooth} and \ref{tab:adv_cs_smooth} show the error tables for the UB and for the coupled scheme respectively. The error tables of the coupled and of the SL coincide on this test since the switching indicator is able to recognise that the solution is smooth enough and there are no jumps.
\begin{center}
	\begin{figure}[!hbt]
		\begin{tabular}{cc}
			\includegraphics[width=0.4\textwidth]{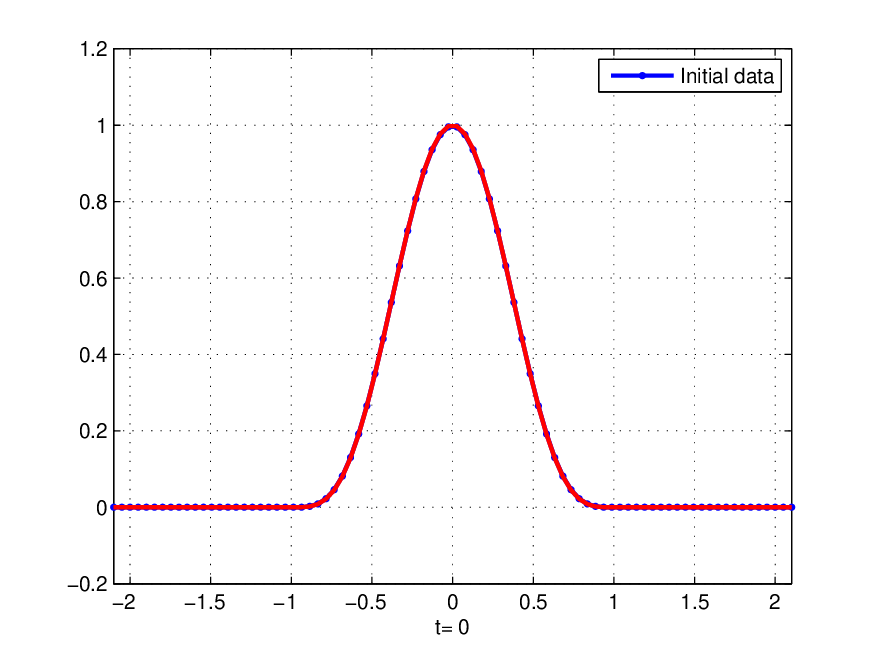}&
			\includegraphics[width=0.4\textwidth]{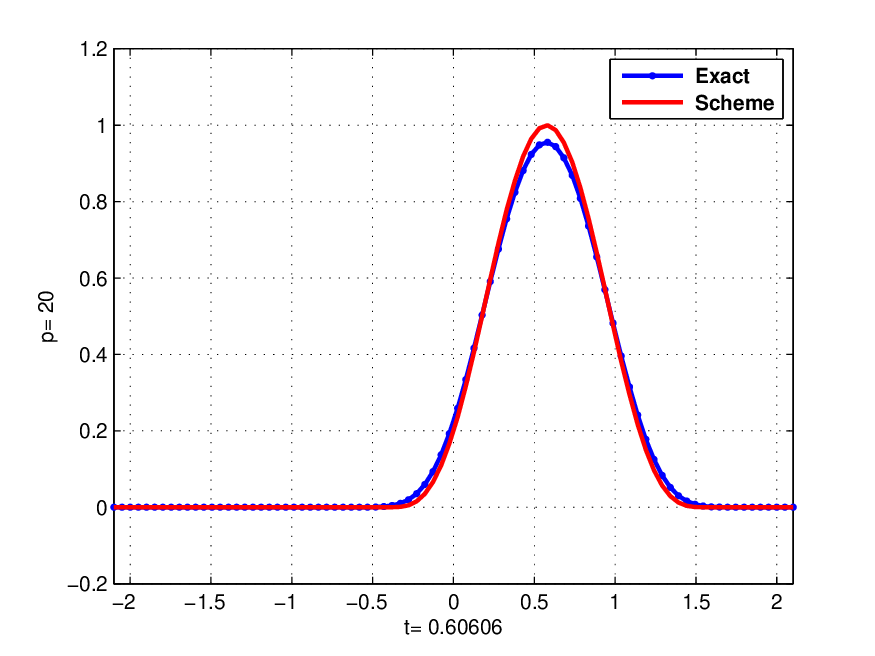}  \\
			\includegraphics[width=0.4\textwidth]{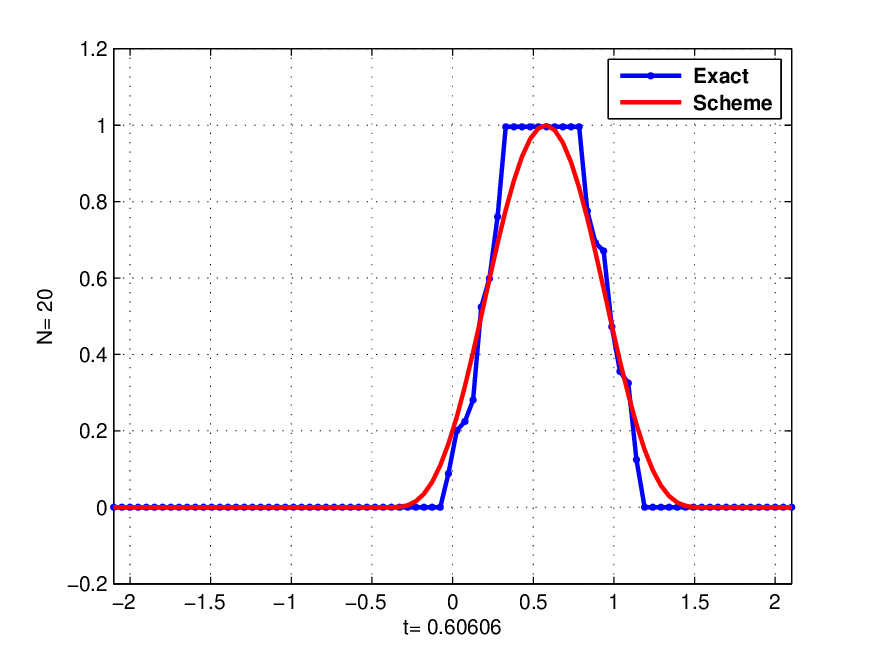} &
			\includegraphics[width=0.4\textwidth]{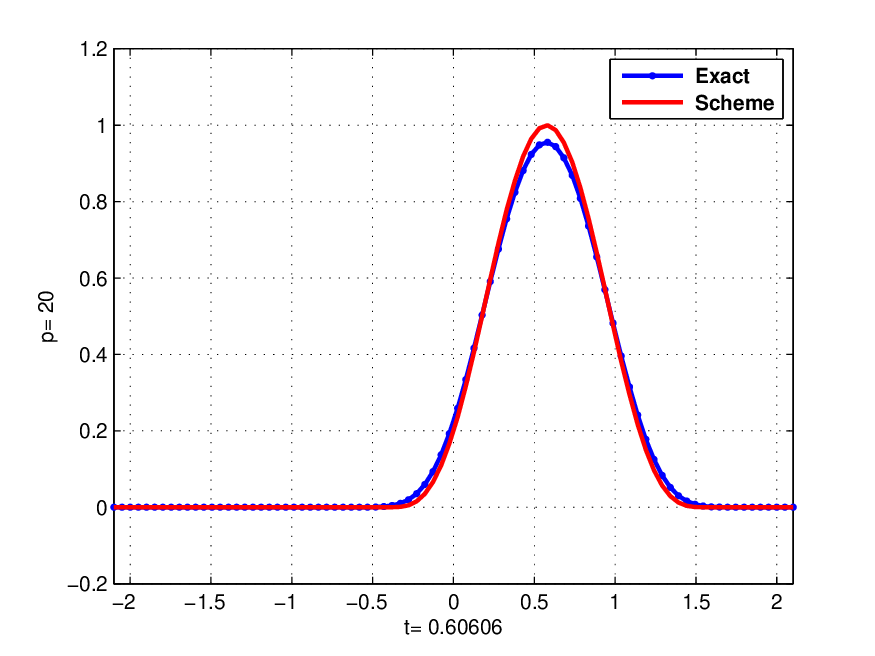}
		\end{tabular}
		\caption{\small Example \ref{ex:adv}, test 1: plots of the solutions for $t=20\Delta t$, $\Delta t=0.045$. 
			Top: initial data \eqref{eq:smooth} (left), \SL~scheme (right) and bottom: \UB~scheme (left), coupled scheme (right).}
		\label{fig:smooth_adv1}
	\end{figure}
	\begin{figure}[!hbt]
		\includegraphics[width=0.3\textwidth]{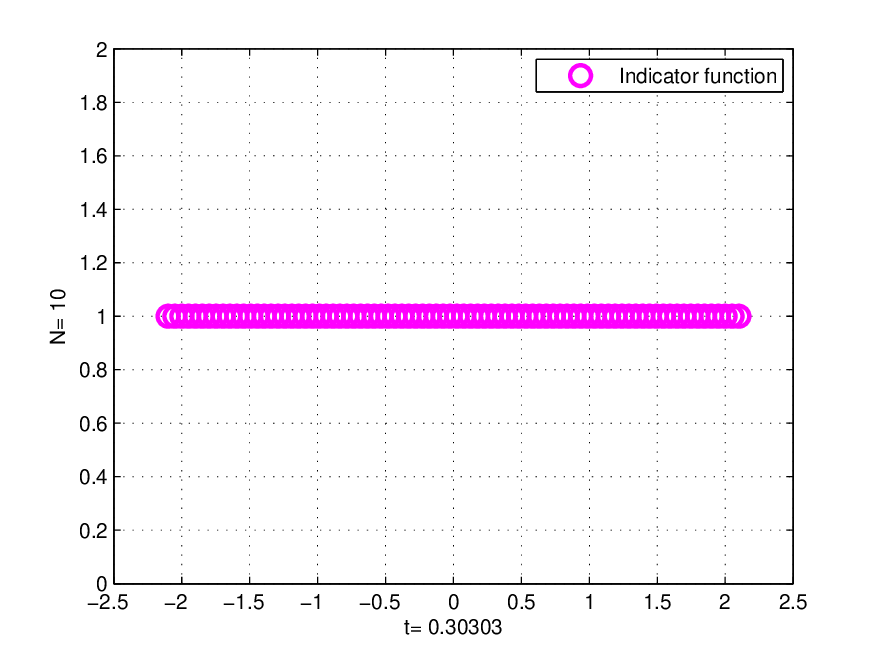}
		\includegraphics[width=0.3\textwidth]{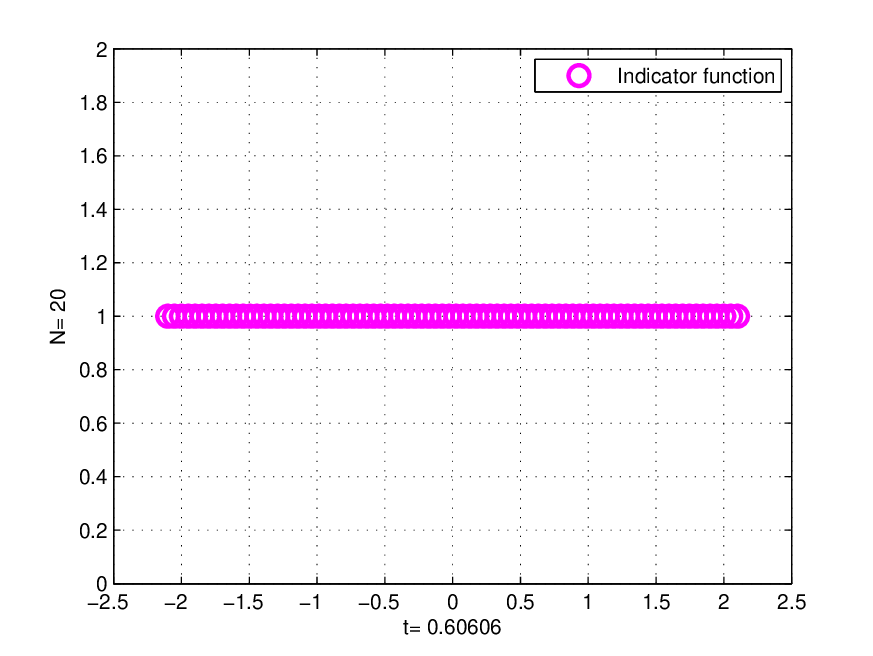}
		\includegraphics[width=0.3\textwidth]{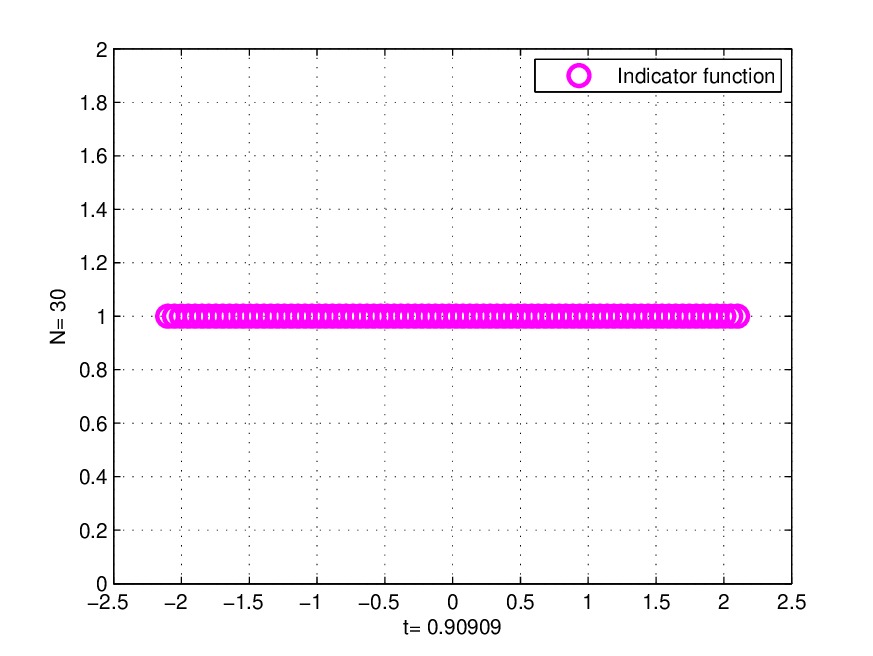}\\
		\caption{\small{Example \ref{ex:adv}, test 1: the plot of the indicator function $\sigma$ at $t=10\Delta t,~20\Delta t,~30\Delta t$ for $\Delta t=0.045.$}}
		\label{fig:sigma_smooth_adv1}
	\end{figure}
\end{center}
\begin{center}
	\begin{table}[!hbt]
\captionsetup{width=1\textwidth}
		\begin{tabular}{|c|c|c|c|c|}
				\hline
				$\Delta t$ & $\Delta x$ &$L^1$ Error & $L^2$ Error& $L^\infty$ Error \\
				\hline
				0.181818 & 0.210526 &9.04E-002&8.49E-002&1.10E-001\\
				0.090909 &0.102564  &4.32E-002&4.27E-002&7.68E-002\\
				0.045455 &0.050633  &2.17E-002&2.22E-002&5.56E-002\\
				0.022472&0.025157  &1.27E-002&1.27E-002&3.59E-002\\
				0.011236 &0.012539  &6.49E-003&6.68E-003&2.20E-002\\
				0.005634 &0.006260  &3.34E-003&3.50E-003&1.09E-002\\
				\hline
		\end{tabular}
	\caption{Example \ref{ex:adv}, test 1: errors for the \UB~scheme with initial condition~\eqref{eq:smooth} at time $T=2$.}
	\label{tab:adv_ub_smooth}
	\end{table}
	\vspace*{-1cm}
	\begin{table}[!hbt]
		\captionsetup{width=1\textwidth}
		\begin{tabular}{|c|c|c|c|c|}
				\hline
				$\Delta t$ & $\Delta x$ &$L^1$ Error & $L^2$ Error& $L^\infty$ Error \\
				\hline
				0.181818 & 0.210526 &7.36E-002&5.93E-002&7.44E-002\\
				0.090909 &0.102564  &3.49E-002&2.84E-002&3.64E-002\\
				0.045455 &0.050633  &1.67E-002&1.37E-002&1.75E-002\\
				0.022472 &0.025157  &8.87E-003&7.28E-003&9.31E-003\\
				0.011236 &0.012539  &4.38E-003&3.60E-003&4.60E-003\\
				0.005634 &0.006260  &2.14E-003&1.76E-003&2.25E-003\\
				\hline
		\end{tabular}\caption{Example \ref{ex:adv}, test 1: errors for the coupled \SL~+ \UB~scheme with initial condition~\eqref{eq:smooth} at time $T=2$.}
		\label{tab:adv_cs_smooth}
	\end{table}
\end{center}
{\em Test 2:  Discontinuous initial condition $v_0$.}
\begin{equation}\label{eq:jump}
	v(0,x)=v_0(x)=\left\{\begin{array}{lll}
		1&\qquad&  \text{if}\quad|x|\leq1\\
		0&\qquad&\textrm{otherwise}.
	\end{array}\right.
\end{equation}
For the piecewise discontinuous initial data UB~scheme is already good so we expect the coupled scheme to switch to UB~scheme.   
Fig.~\ref{fig:jump_adv}, shows the solution of~\eqref{eq:adv} at time $t=20\Delta t$ with $\Delta t=0.045$ for the initial data~\eqref{eq:jump}. Fig.~\ref{fig:sigma_jump_adv}, shows the plots of $\sigma$ for different times $t=10\Delta t,~20\Delta t,~30\Delta t$ where $\Delta t=0.45$. In Fig.~\ref{fig:sigma_jump_adv}, we can see that $\sigma=0$ everywhere and hence coupled scheme is the same as \UB~scheme. Moreover one can see that once the scheme switches to the UB scheme ($\sigma=0$) it keeps that choice on a larger number of cells in order to follow the jump. Tables ~\ref{tab:adv_ub_jump} and~\ref{tab:adv_sl_jump}, show the error tables respectively for the coupled, SL scheme respectively. Here coupled scheme is same as UB scheme.
\begin{center}
\begin{figure}[!hbt]
		\begin{tabular}{cc}
			\includegraphics[width=0.4\textwidth]{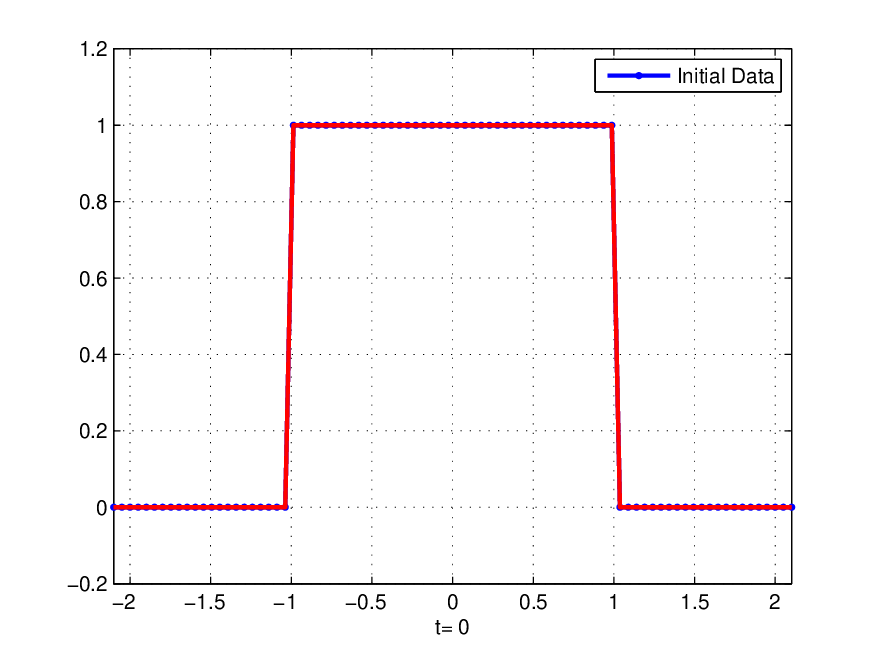}&
			\includegraphics[width=0.4\textwidth]{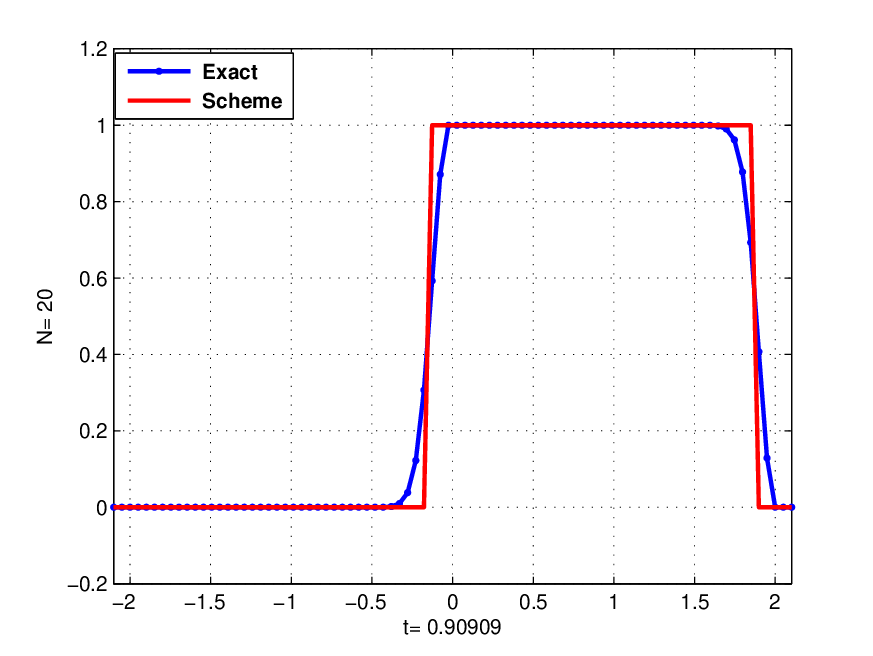}\\
			\includegraphics[width=0.4\textwidth]{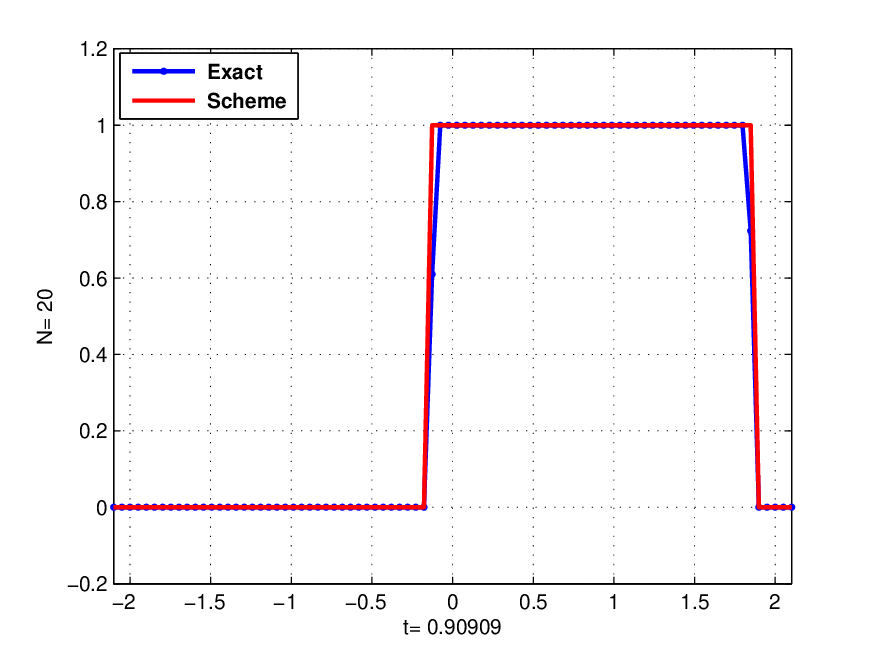} &
			\includegraphics[width=0.4\textwidth]{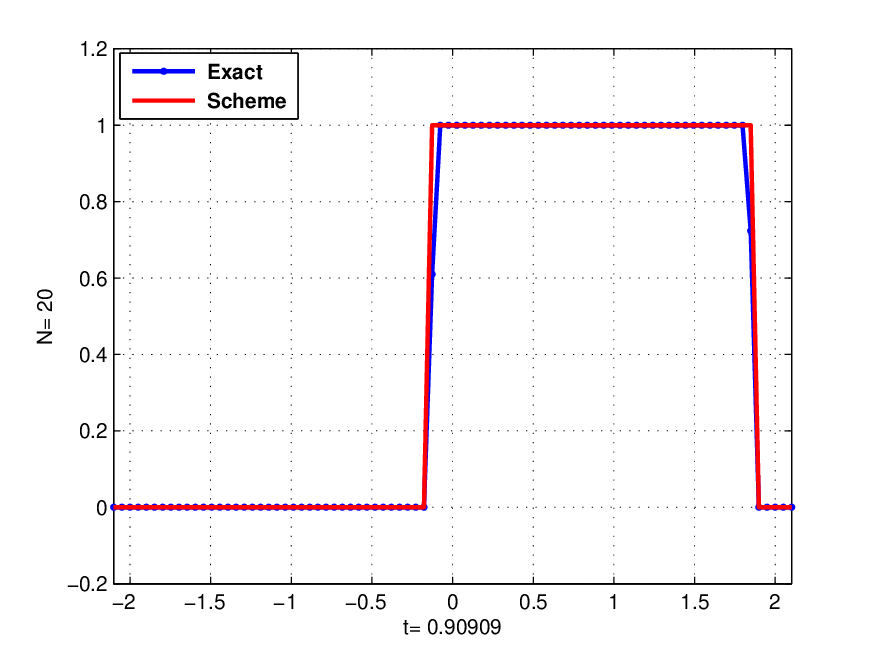}
		\end{tabular}
		\caption{\small Example \ref{ex:adv}, test 2: plots of the solutions for $t=20\Delta t$, $\Delta t=0.045$. 
			Top: initial data~\eqref{eq:jump} (left), \SL~scheme (right).  Bottom: \UB~scheme (left), coupled scheme (right).}
		\label{fig:jump_adv}
	\end{figure}
	\end{center}	
\begin{center}
	\begin{figure}[!hbt]
		\centering
		\includegraphics[width=0.3\textwidth]{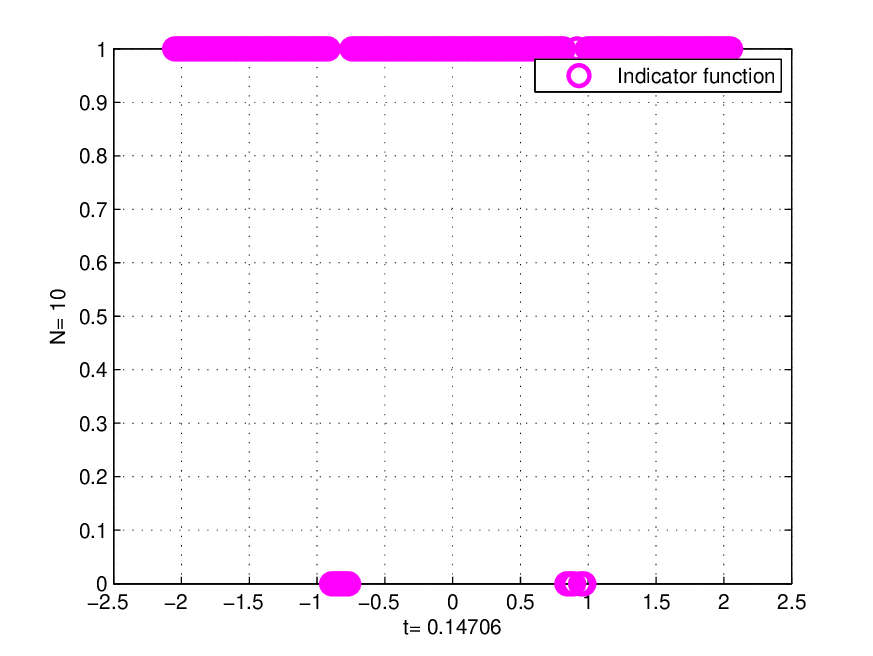}
		\includegraphics[width=0.3\textwidth]{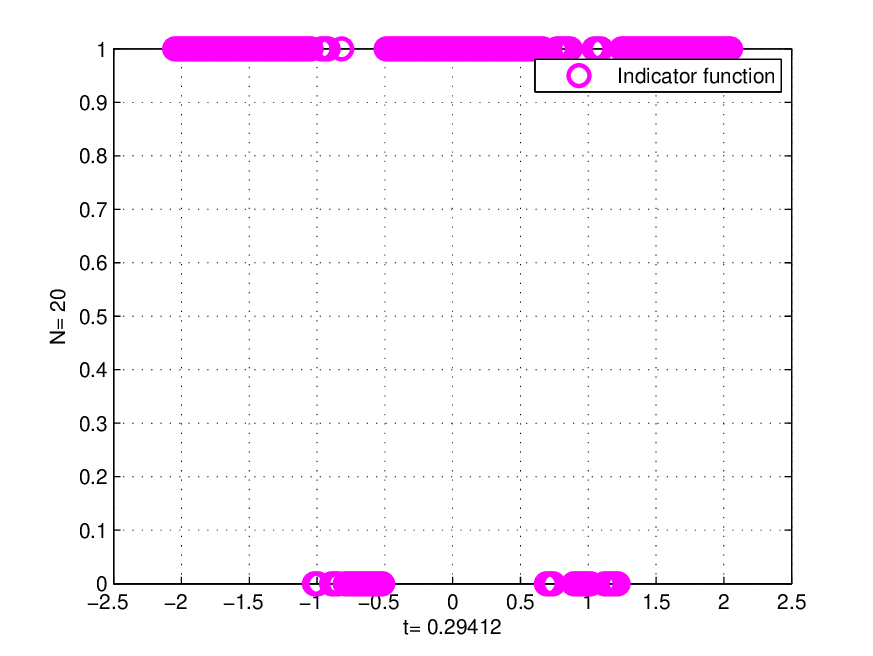}
		\includegraphics[width=0.3\textwidth]{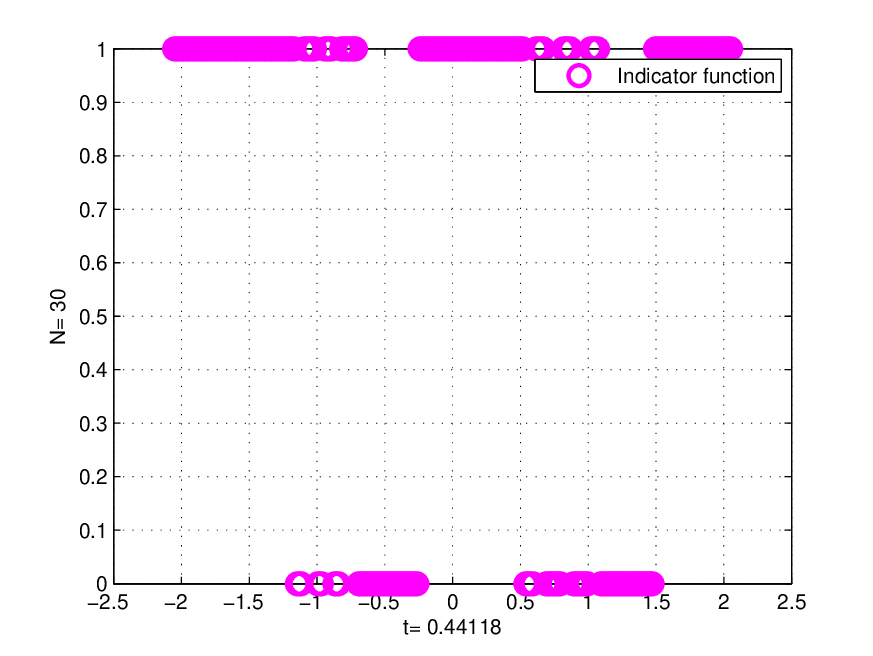}\\
		\caption{Example \ref{ex:adv}, test 2: the plot of the indicator function $\sigma$ for \eqref{eq:jump} at $t=10\Delta t,~20\Delta t,~30\Delta t$ where $\Delta t=0.045$.}  
		\label{fig:sigma_jump_adv}
	\end{figure}
\end{center}	
	\begin{table}[!hbt]
	\captionsetup{width=1\textwidth}
		\begin{tabular}{|c|c|c|c|c|}
				\hline
				$\Delta t$ & $\Delta x$ &$L^1$ Error & $L^2$ Error& $L^\infty$ Error \\
				\hline
				0.181818 & 0.210526 &7.02E-002&1.12E-001&2.12E-001\\
				0.090909 &0.102564  &6.84E-002&1.51E-001&3.56E-001\\
				0.045455 &0.050633  &3.38E-002&1.08E-001&3.90E-001\\
				0.022472 &0.025157  &1.68E-002&8.32E-002&4.96E-001\\
				0.011236 &0.012539  &8.36E-003&6.24E-002&5.44E-001\\
				0.005634 &0.006260  &4.17E-003&3.73E-002&3.33E-001\\
				\hline
		\end{tabular}\caption{Example~\ref{ex:adv}, test 2: errors for the \SL+\UB~ scheme with initial condition~\eqref{eq:jump} and $T=2$.}
		\label{tab:adv_ub_jump}
	\end{table}
	\begin{table}[!hbt]
			\captionsetup{width=1\textwidth}	
		\begin{tabular}{|c|c|c|c|c|}
				\hline
				$\Delta t$ & $\Delta x$ &$L^1$ Error & $L^2$ Error& $L^\infty$ Error \\
				\hline
				0.181818 & 0.210526 &2.53E-001&3.32E-001&5.56E-001\\
				0.090909 &0.102564  &1.43E-001&1.99E-001&3.38E-001\\
				0.045455 &0.050633  &1.03E-001&1.72E-001&4.07E-001\\
				0.022472 &0.025157  &7.57E-002&1.47E-001&4.05E-001\\
				0.011236 &0.012539  &5.37E-002&1.25E-001&4.38E-001\\
				0.005634 &0.006260  &3.77E-002&1.05E-001&4.61E-001\\
				\hline
		\end{tabular}\caption{Example~\ref{ex:adv}, test 2: errors for the \SL~scheme with initial condition~\eqref{eq:jump} at time $T=2$.}
		\label{tab:adv_sl_jump}
	\end{table}
\end{example}
\begin{example}\label{ex:adv_mix}{\bf Advection with constant velocity and mixed initial conditions.}\\
We consider the advection equation \eqref{eq:adv} with $c\equiv 0.1$ and $v_0(x)$  is an initial condition with compact support. In particular, we take an initial condition $v_0$  which contains three bumps
\begin{equation}\label{eq:mix}
v(0,x)=v_0(x)=\left\{\begin{array}{lll}
1-|x+3|&\quad\quad& -4 < x < -2\\
(1-x^2)^4&\quad\quad&  -1< x < 1\\
1&\quad\quad&   2 <x < 3\\
0&\quad\quad& \textrm{otherwise},
\end{array}\right.
\end{equation}	
In this test $\Omega=(-4.5,4.5)$, $T=6$ and $\nu=0.0833$. 
Fig.~\ref{fig:mixed} compares the different plots. Note that the SL~scheme has a rather big error around the jumps, even where the solution is flat and has a larger support with respect to the exact solution. The UB~has the typical piecewise constant behavior in the regularity region but keeps the support correctly. In Table~\ref{tab:adv_CS_mix} we show the errors for $L^1(\Omega)$, $L^2(\Omega)$, $L^\infty (\Omega)$ and $L^\infty (\Omega_{reg})$. It is clear that in this example initial data $v_0$~\eqref{eq:mix} have different regularity in different intervals so we expect to get small $L^\infty$ errors only in the domain where the solution is regular,  outside even a single node can make the $L^\infty$ error increase. That is why that $L^1 (\Omega)$ errors for the coupled scheme are close to those of the \UB~and  the local $L^\infty (\Omega_{reg})$ error are better with respect to the SL~scheme and UB~scheme. 
\begin{center}
	\begin{figure}[!hbtp]
		\centering
		\begin{tabular}{cc}
			\includegraphics[width=0.4\textwidth]{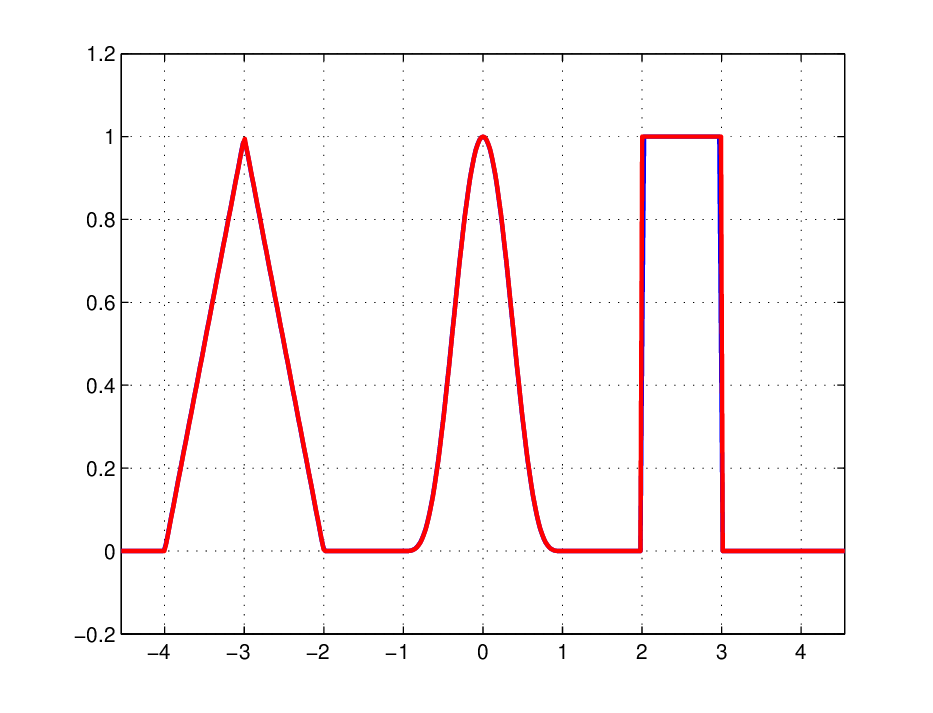}&	\includegraphics[width=0.4\textwidth]{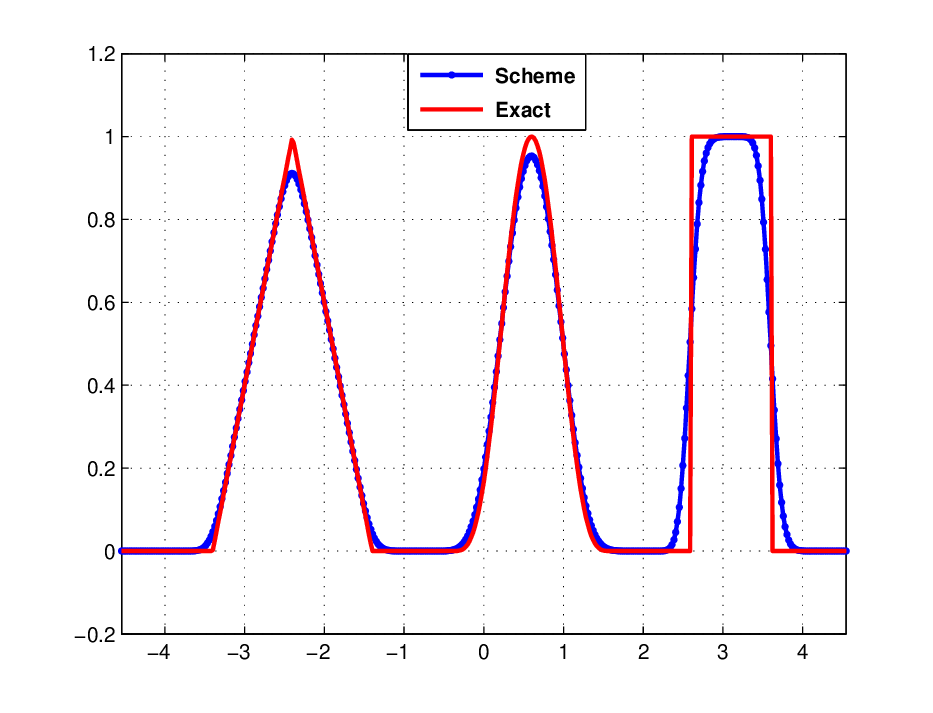}  \\				
			\includegraphics[width=0.4\textwidth]{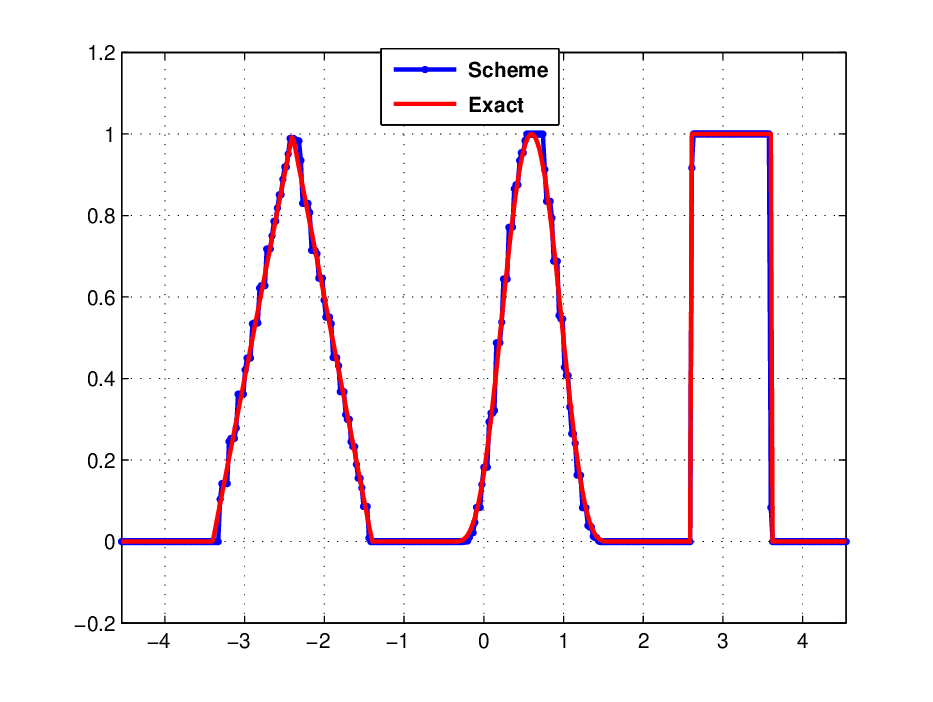} & \includegraphics[width=0.4\textwidth]{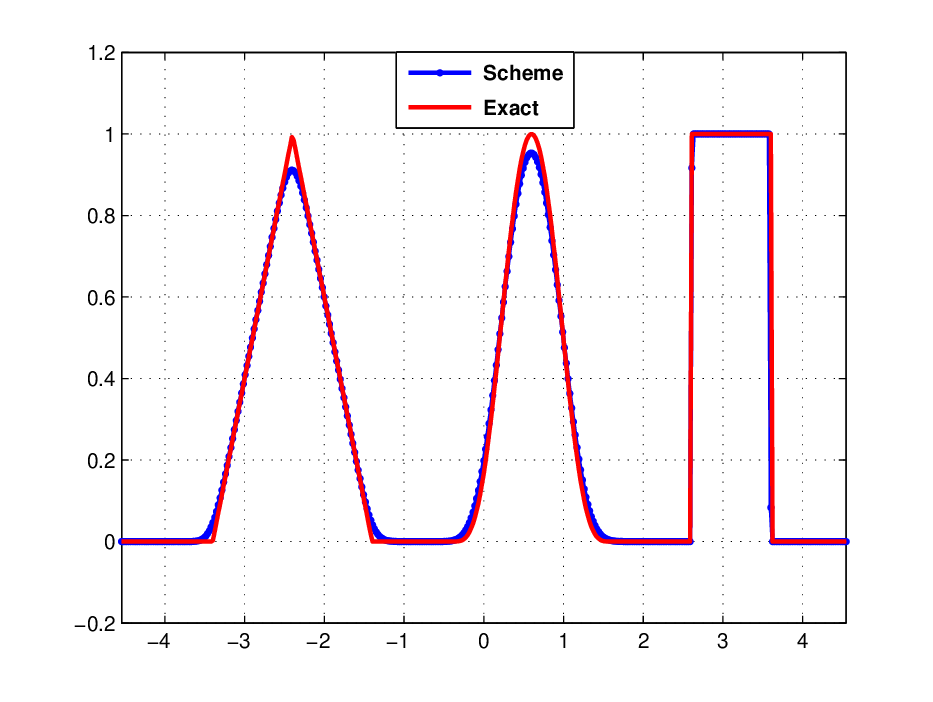}	
		\end{tabular}
		\caption{Example~\ref{ex:adv_mix}, plot of the solutions at  $T=1$. Top: initial data~\eqref{eq:mix} (left)  and  of  \SL~scheme (right). Bottom: \UB~(left), coupled scheme (right).} 
		\label{fig:mixed}
	\end{figure}
	\begin{table}[htp]
				\captionsetup{width=1\textwidth}
		\begin{tabular}{|c|c|c|c|c|}
				\hline
				$\Delta t$ & $\Delta x$ &$L^1$ Error & $L^2$ Error& $L^\infty$ Error \\
				\hline
				0.0750 &  0.0900 & 4.16E-001 & 3.69E-001 & 9.16E-001\\
				0.0375 &  0.0450 & 1.79E-001 & 1.27E-001 & 2.50E-001\\
				0.0187 &  0.0225 & 9.82E-002 & 1.46E-001 & 9.16E-001\\
				0.0094 &  0.0112 & 4.87E-002 & 8.87E-002 & 7.50E-001\\
				0.0047 &  0.0056 & 2.34E-002 & 7.01E-002 & 9.16E-001\\		
				\hline
		\end{tabular}\caption{Example~\ref{ex:adv_mix}, errors for \UB~scheme at time $T=6$.}
		\label{tab:adv_mix}
	\end{table}				
		\begin{table}[htp]
				\captionsetup{width=1\textwidth}
		\begin{tabular}{|c|c|c|c|c|}
				\hline
				$\Delta t$ & $\Delta x$ &$L^1$ Error & $L^2$ Error& $L^\infty$ Error \\
				\hline
				0.0750 &  0.0900 & 6.14E-001& 3.61E-001 & 5.09E-001\\
				0.0375 &  0.0450 & 3.99E-001& 2.89E-001 & 5.44E-001\\
				0.0187 &  0.0225 & 2.49E-001& 2.34E-001 & 5.04E-001\\
				0.0094 &  0.0112 & 1.62E-001& 1.94E-001 & 5.22E-001\\
				0.0047 &  0.0056 & 1.07E-001& 1.62E-001 & 5.02E-001 \\
				\hline
		\end{tabular}\caption{Example~\ref{ex:adv_mix}, errors for \SL~scheme at time $T=6$.}
		\label{tab:adv_ub_mix}
	\end{table}	
	\begin{table}[htp]
				\captionsetup{width=1\textwidth}
\begin{tabular}{|c|c|c|c|c|c|}
				\hline
				$\Delta t$ & $\Delta x$ &$L^1$ Error & $L^2$ Error& $L^\infty$ Error& $L^\infty (\Omega_{reg})$ \\
				\hline
				0.0750 &  0.0900 & 3.49E-001& 3.14E-001 & 9.16E-001 & 1.61E-001\\
				0.0375 &  0.0450 & 1.60E-001& 1.13E-001 & 2.50E-001 & 1.12E-001\\
				0.0187 &  0.0225 & 9.41E-002& 1.46E-001 & 9.16E-001 & 8.12E-002\\
				0.0094 &  0.0112 & 4.79E-002& 8.78E-002 & 7.50E-001 & 5.92E-002\\
				0.0047 &  0.0056 & 2.41E-002& 7.05E-002 & 9.16E-001 & 4.25E-002\\
				\hline
		\end{tabular}\caption{Example~\ref{ex:adv_mix},  errors for coupled \SL + \UB~scheme  at $T=6$.}
		\label{tab:adv_CS_mix}
	\end{table}		
\end{center}
\end{example}
\begin{example}\label{ex:adv_var}{\bf Advection equation with variable velocity.}\\
In this example, we consider the advection equation 
with the variable velocity $c(x)=-(x-\bar{x})$, where $\bar{x}=1.1$ (this example has been taken from ~\cite{FF14}). We consider
smooth initial data which has bounded second derivative i.e.
\begin{equation}\label{eq:smooth_var}
v(0,x)=v_0(x)=\max(0, 1-16(x-0.25)^2)^2
\end{equation}
Here the domain $\Omega=(0,1)$, $T=1$ and  we fix $\nu= 0.6$. As the solution is smooth, we expect our coupled scheme to coincide with SL~everywhere. Fig.~\ref{fig:smooth_var} shows the solution corresponding to~\eqref{eq:smooth_var} at time $t=20\Delta t$ with the time step $\Delta t=0.015385$ for the different schemes. 
As one can see the UB~scheme has its typical stepwise approximation. Fig.~\ref{fig:sigma_smooth_adv_var} shows the plot of $\sigma$ for different times $t=10\Delta t,~20\Delta t,~30\Delta t~$. The switching indicator is able to detect that the solution is smooth and chooses to apply the SL scheme in the whole domain.
Comparing Tables \ref{tab:adv_var_ub_smooth} and Table \ref{tab:adv_var_cs_smooth} it can be seen that in this example the $L^\infty$ error is much better for the SL (and coupled) scheme in particular for larger space/time steps and that the $L^1$ error is of the same order for UB and the coupled scheme.
\vspace*{0.5cm}
\begin{center}
	\begin{figure}[!hbt]
		\begin{tabular}{cc}
			\includegraphics[width=0.4\textwidth]{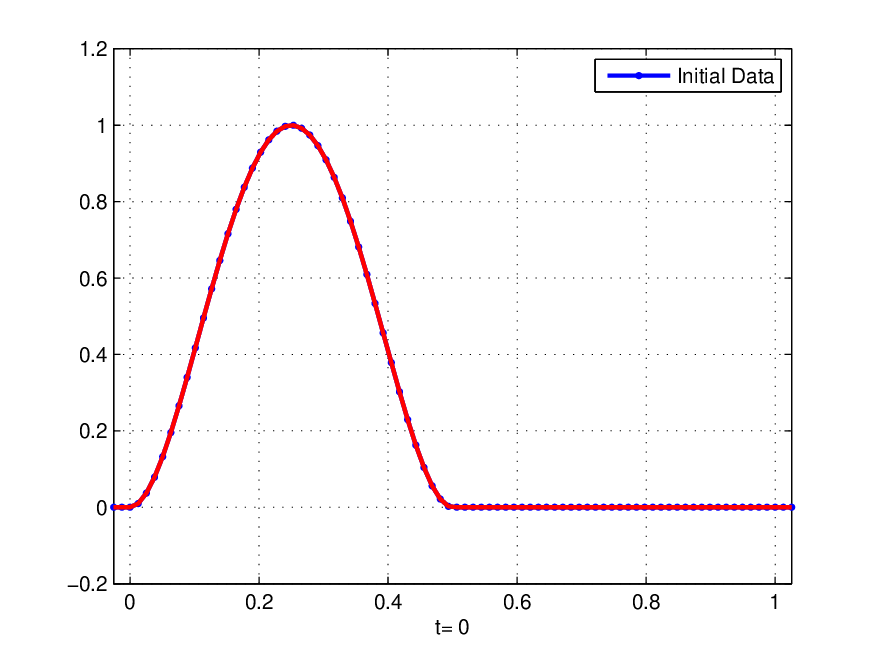}
			\includegraphics[width=0.4\textwidth]{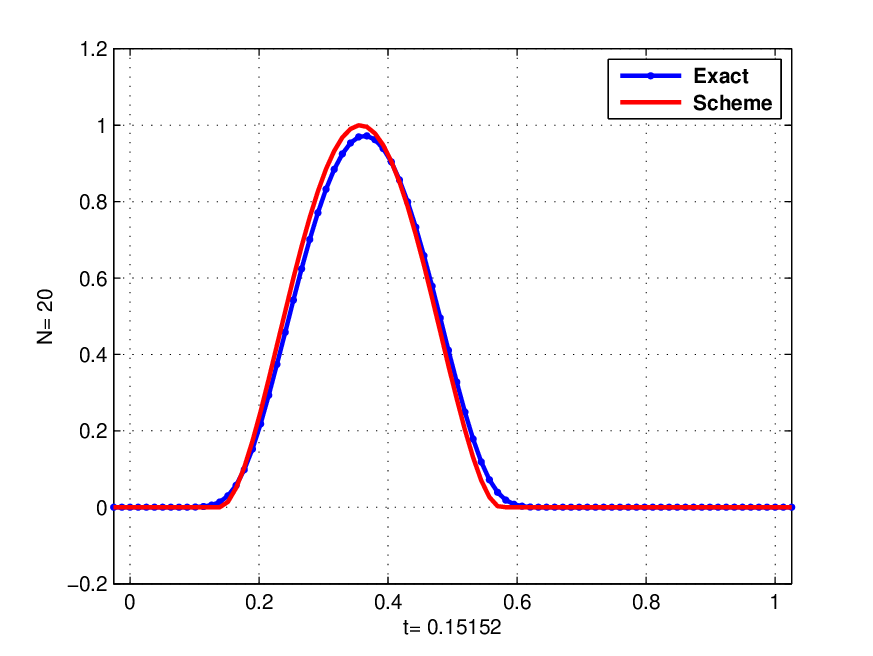}\\
			\includegraphics[width=0.4\textwidth]{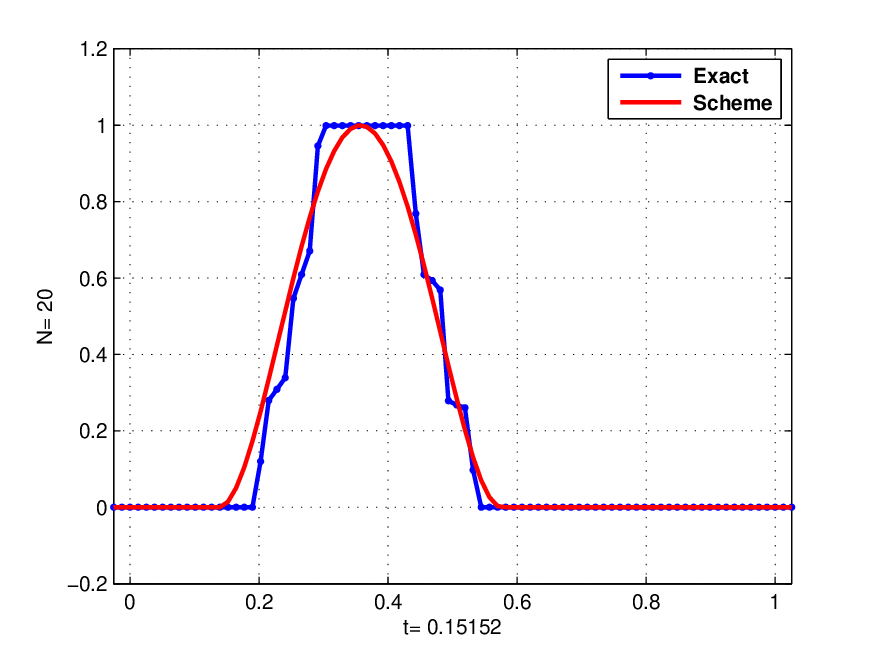}
			\includegraphics[width=0.4\textwidth]{figs/sl_smooth_var}\\
		\end{tabular}
		\caption{\small{Example \ref{ex:adv_var}: plot of the solutions at $t=20\Delta t$ where $\Delta t=0.015385$. Top: initial condition (left) and \SL (right). Bottom: \UB~(left) and coupled scheme (right)}} 
		\label{fig:smooth_var}
	\end{figure} 
\end{center}
\begin{center}
	\begin{figure}[!hbt]
		\includegraphics[width=0.3\textwidth]{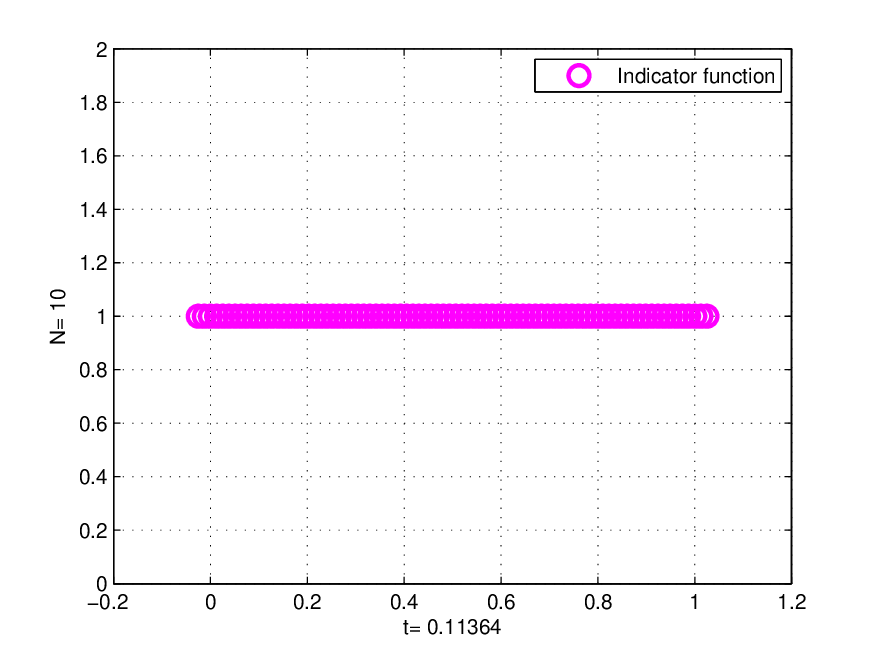}
		\includegraphics[width=0.3\textwidth]{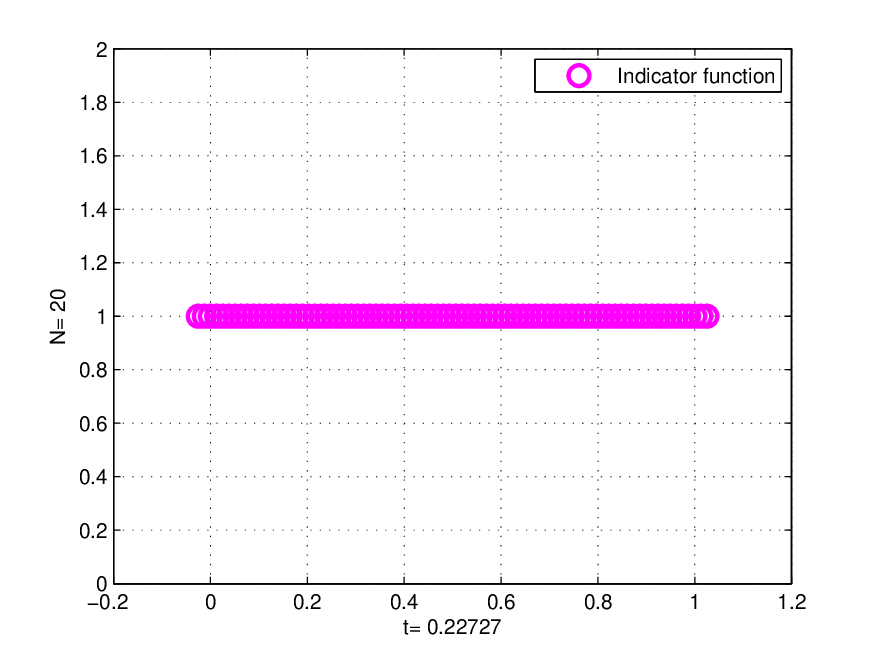}
		\includegraphics[width=0.3\textwidth]{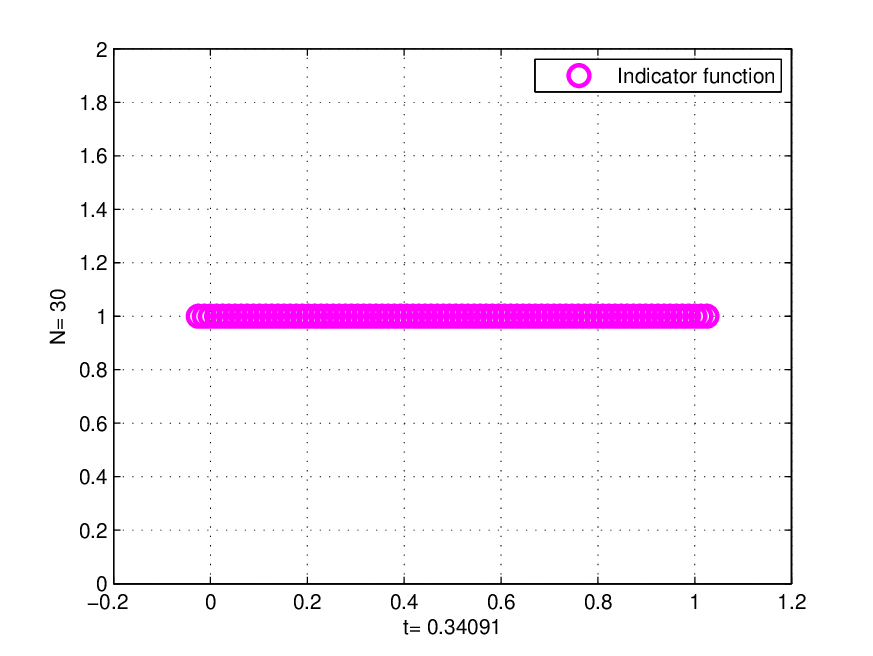}\\
		\caption{\small{Example \ref{ex:adv_var}, the plot of the indicator function $\sigma$ for \eqref{eq:smooth_var} initial data for $\Delta t=0.015385$ and $t=5\Delta t,~10\Delta t,~20\Delta t$.}} 
		\label{fig:sigma_smooth_adv_var}
	\end{figure}
\end{center}
	\begin{table}[!hbt]
				\captionsetup{width=1\textwidth}
		\begin{tabular}{|c|c|c|c|c|}
				\hline
				$\Delta t$ & $\Delta x$ &$L^1$ Error & $L^2$ Error& $L^\infty$ Error \\
				\hline
				0.031250 & 0.052632 & 5.07E-002 & 1.04E-001 & 3.04E-001 \\
				0.015385 & 0.025641 & 5.98E-002 & 1.15E-001 & 3.44E-001 \\
				0.007576 & 0.012658 & 3.08E-002 & 5.85E-002 & 2.12E-001 \\
				0.003774 & 0.006289 & 1.95E-002 & 3.72E-002 & 1.33E-001 \\
				0.001880 & 0.003135 & 1.58E-002 & 3.01E-002 & 1.18E-001 \\
				0.000939 & 0.001565 & 1.54E-002 & 2.74E-002 & 9.67E-002 \\
				\hline
		\end{tabular}\caption{Example \ref{ex:adv_var}, errors for the \UB~scheme with initial condition \eqref{eq:smooth_var} at time $T=1$.}
		\label{tab:adv_var_ub_smooth}
	\end{table}	
\begin{table}[!hbt]
			\captionsetup{width=1\textwidth}
	\begin{tabular}{|c|c|c|c|c|}
			\hline
			$\Delta t$ & $\Delta x$ &$L^1$ Error & $L^2$ Error& $L^\infty$ Error \\
			\hline
			0.031250 & 0.052632 & 2.98E-002 & 4.50E-002 & 9.51E-002\\
			0.015385 & 0.025641 & 1.95E-002 & 3.09E-002 & 6.63E-002\\
			0.007576 & 0.012658 & 1.53E-002 & 2.52E-002 & 5.86E-002\\
			0.003774 & 0.006289 & 1.52E-002 & 2.56E-002 & 6.16E-002\\
			0.001880 & 0.003135 & 1.50E-002 & 2.53E-002 & 6.16E-002\\
			0.000939 & 0.001565 & 1.51E-002 & 2.57E-002 & 6.29E-002\\
			\hline
	\end{tabular}\caption{Example \ref{ex:adv_var}, errors for the coupled \SL+\UB~scheme with initial condition \eqref{eq:smooth_var} at time $T=1$.}
	\label{tab:adv_var_cs_smooth}
\end{table}
\end{example}
\begin{example}\label{ex:hj}
Finally, let us consider the evolutive  Hamilton-Jacobi equation
\begin{equation}\label{eq;hje}
v_t+|c v_x|=0\quad (t,x)\in \Omega.
\end{equation}
We take $c=1$ and the smooth initial condition \eqref{eq:smooth}. Here the domain $\Omega=[-2,2]$, $T=0.5$ and we fix $\nu=0.6$. In the tables, all the errors are global in $\Omega$.
In this case, the initial solution is smooth but at some point, the solution loses its regularity and a kink appears at $x=0$. So at the beginning we expect the coupled scheme to apply the SL scheme and when at time $t_n$ the singularity is detected the switching parameter $\sigma^n_j$ becomes 0 in a cell and the coupled scheme must switch to the UB~scheme. Fig.~\ref{fig:smooth_sl_hj}--\ref{fig:smooth_cs_hj}, show the plots at different time steps of the evolution of the same initial condition \eqref{eq:smooth}. Fig.~\ref{fig:sigma_smooth_hj}, shows the evolution of the switching indicator which has the desired behavior, i.e. until $t=10 \Delta t$ solution is smooth (and $\sigma\equiv 1$ everywhere).  After that time when the regularity is lost  a singularity is detected and $\sigma=0$ at $x=0$, so the scheme switches correctly to the \UB~scheme. Table~\ref{tab:hj_sl_smooth}--\ref{tab:hj_ub_smooth}, show the error tables of SL, UB and coupled scheme respectively at time $t=20\Delta t$ for $\Delta t=0.014706$.
It should be noted that in this example the coupled scheme is always more accurate with respect to the UB scheme and in general its accuracy is very close to the SL scheme, for small space/time steps it is almost identical to the SL scheme. This is probably due to the fact that the singularity stays at $x=0$ and that the solution is still Lipschitz continuous.
\begin{center}
	\begin{figure}[!hbt]
		\begin{tabular}{cc}
			\includegraphics[width=0.4\textwidth]{figs/initial_smooth}&
			\includegraphics[width=0.4\textwidth]{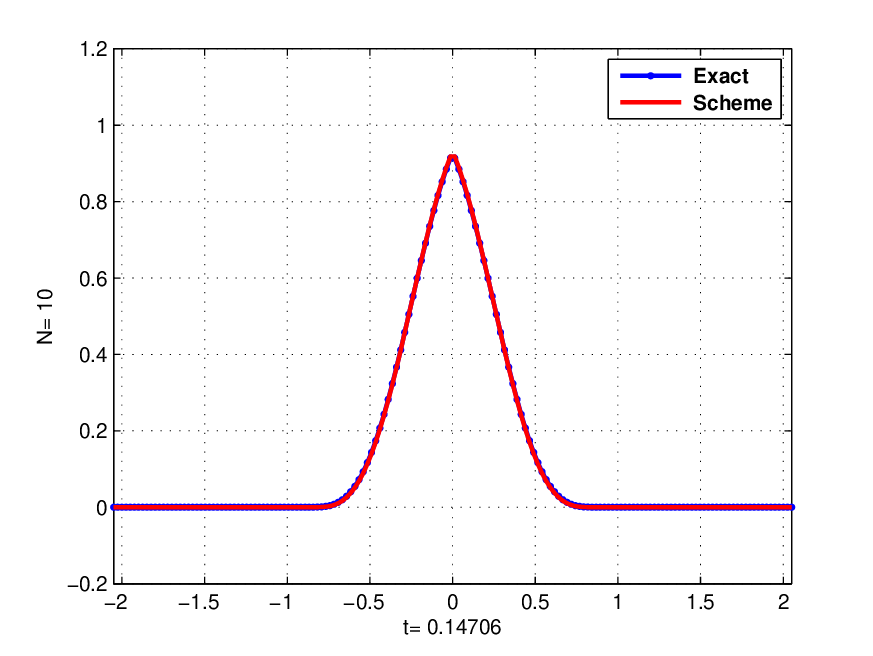}\\
			\includegraphics[width=0.4\textwidth]{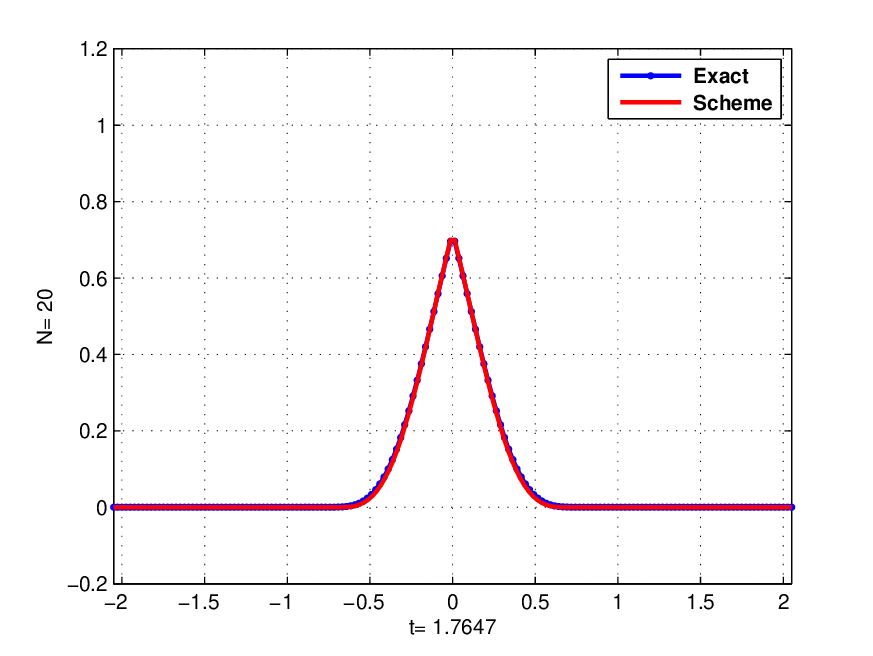}&
			\includegraphics[width=0.4\textwidth]{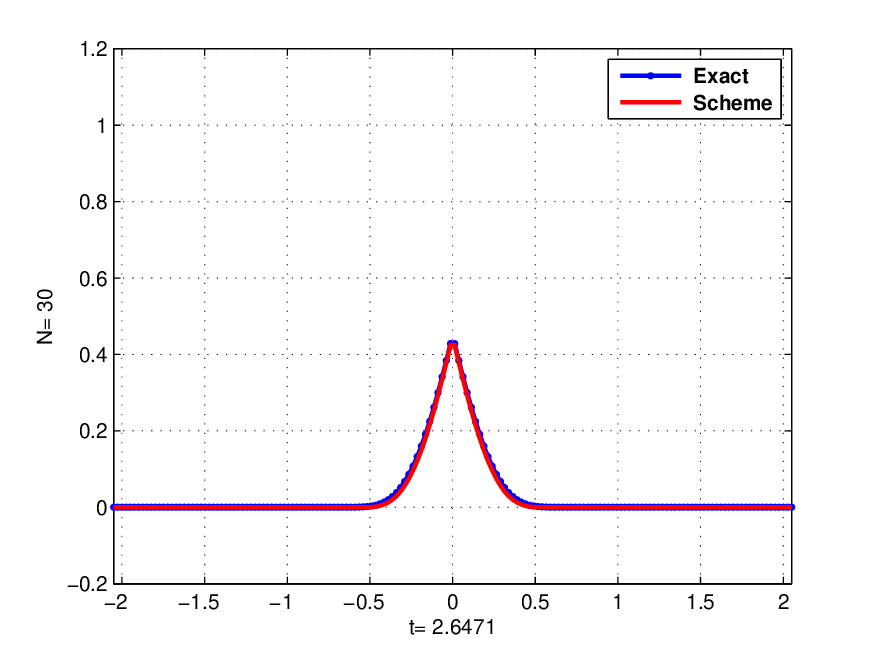}\\
		\end{tabular}
		\caption{Example~\ref{ex:hj}, evolution of the initial condition \eqref{eq:smooth} for the \SL~scheme at  $t=10\Delta t,~20\Delta t,~30\Delta t$, where $\Delta t=0.014706$.} 
		\label{fig:smooth_sl_hj}
	\end{figure} 
	\begin{figure}[!hbt]
		\begin{tabular}{cc}
			\includegraphics[width=0.4\textwidth]{figs/initial_smooth}&
			\includegraphics[width=0.4\textwidth]{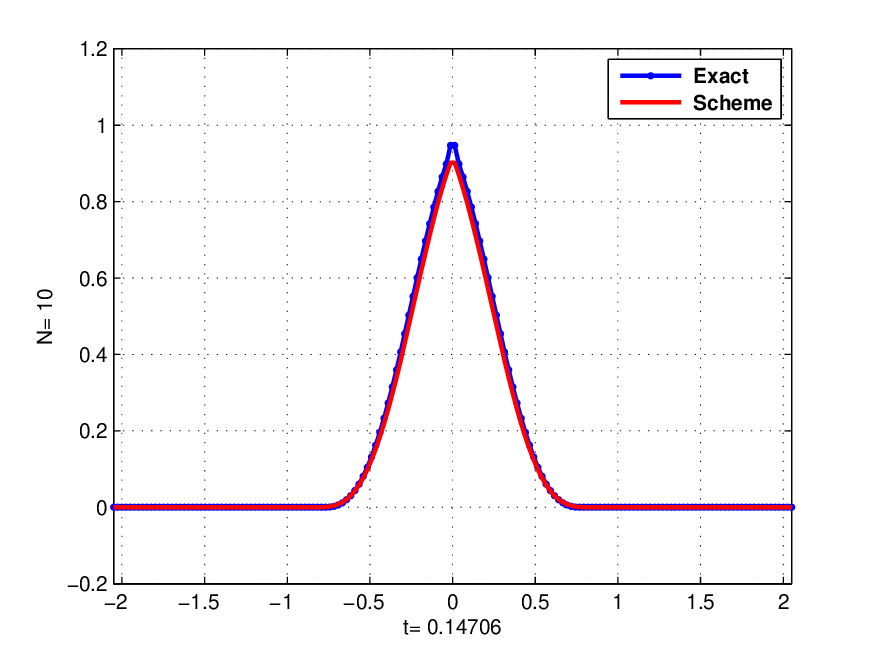}\\
			\includegraphics[width=0.4\textwidth]{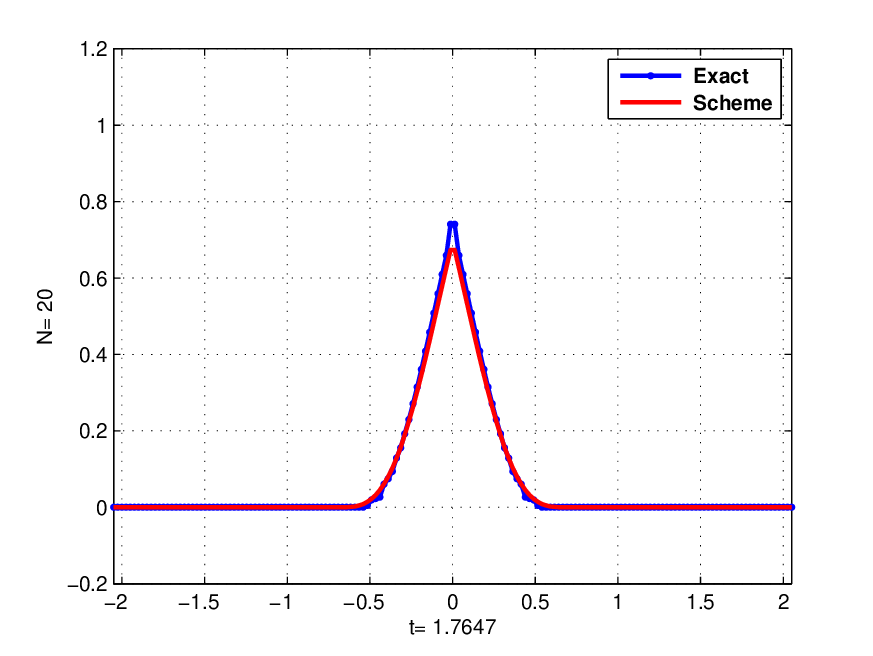}&
			\includegraphics[width=0.4\textwidth]{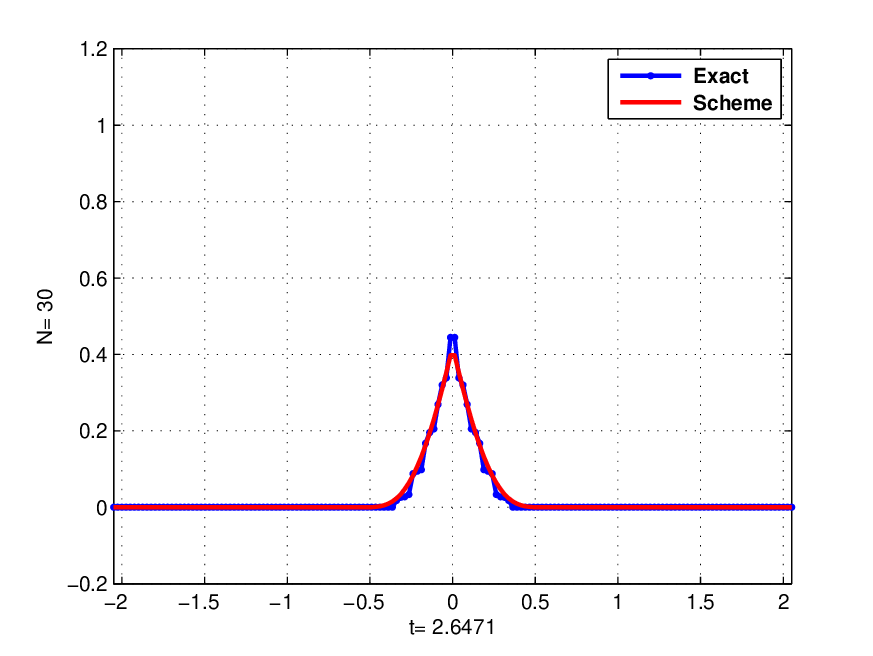}\\
		\end{tabular}
		\caption{Example~\ref{ex:hj}: evolution of the initial condition \eqref{eq:smooth} for the \UB~scheme at  $t=10\Delta t,~20\Delta t,~30\Delta t$,  where $\Delta t=0.014706$.} 
		\label{fig:smooth_ub_hj}
	\end{figure} 
	\begin{figure}[!hbt]
		\begin{tabular}{cc}
			\includegraphics[width=0.4\textwidth]{figs/initial_smooth}&
			\includegraphics[width=0.4\textwidth]{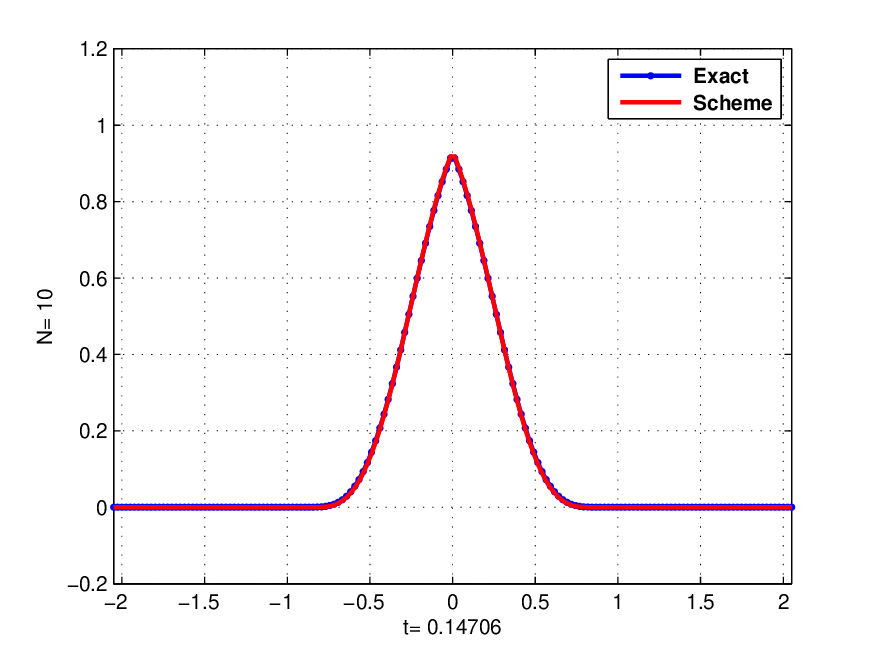}\\
			\includegraphics[width=0.4\textwidth]{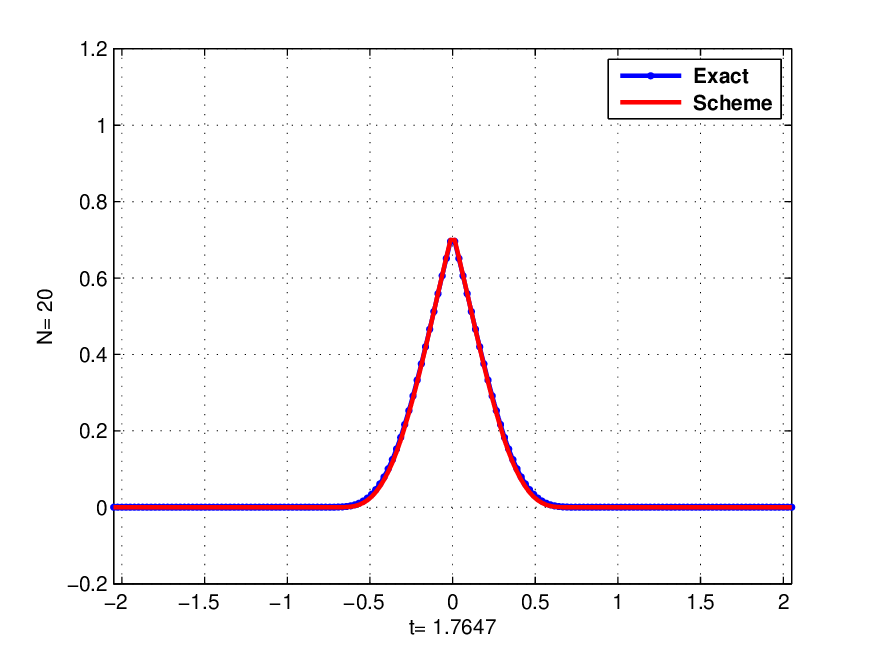}&
			\includegraphics[width=0.4\textwidth]{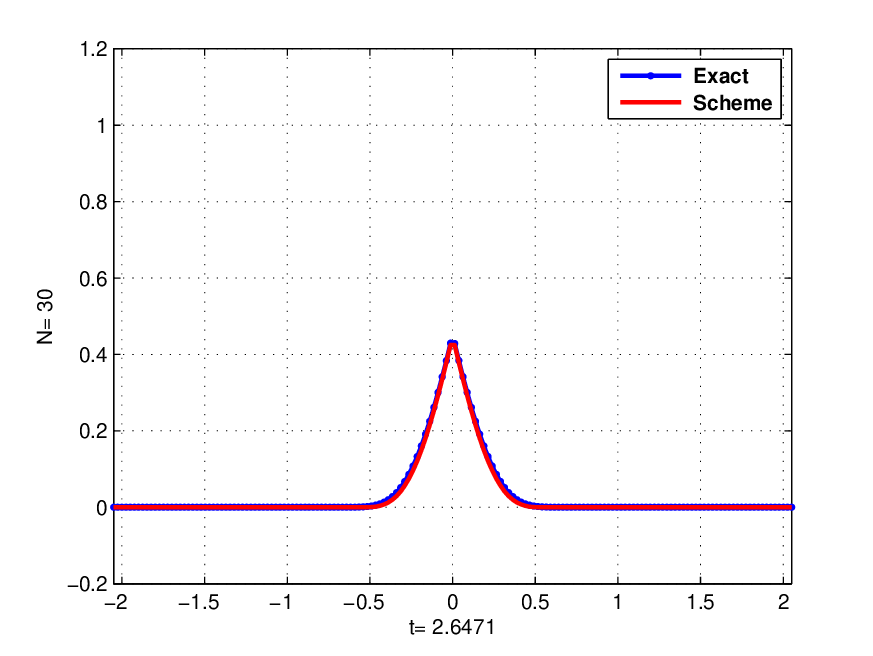}\\
		\end{tabular}
		\caption{Example~\ref{ex:hj}, evolution of the initial condition \eqref{eq:smooth} for the coupled \SL+\UB~scheme at  $t=10\Delta t,20\Delta t,30\Delta t$ where $\Delta t=0.014706$.} 
		\label{fig:smooth_cs_hj}
	\end{figure} 
\end{center}
\begin{center}
	\begin{figure}[!hbt]
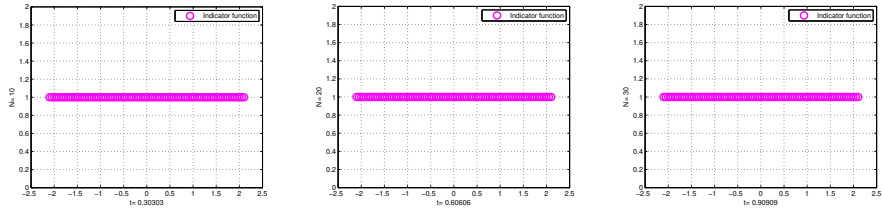

		\includegraphics[width=0.3\textwidth]{figs/sigma_smooth_10}
		\includegraphics[width=0.3\textwidth]{figs/sigma_smooth_20}
		\includegraphics[width=0.3\textwidth]{figs/sigma_smooth_30}\\
		\caption{Example~\ref{ex:hj}, plot of the indicator function $\sigma$ for \eqref{eq:smooth} initial data at $t=10\Delta t,~20\Delta t,~30\Delta t$, $\Delta t=0.014706$.}  
		\label{fig:sigma_smooth_hj}
	\end{figure}
\end{center}
\begin{center}
	\begin{table}[!hbt]
				\captionsetup{width=1\textwidth}
\begin{tabular}{|c|c|c|c|c|}
				\hline
				$\Delta t$ & $\Delta x$ &$L^1$ Error & $L^2$ Error& $L^\infty$ Error \\
				\hline
				0.125000&0.210526&2.34E-001&2.42E-001&3.28E-001\\ 
				0.055556&0.102564&8.26E-002&1.27E-001&2.69E-001\\ 
				0.029412&0.050633&3.86E-002&5.20E-002&1.30E-001\\ 
				0.014706&0.025157&1.76E-002&2.27E-002&6.69E-002 \\
				0.007463&0.012539&8.59E-003&1.04E-002&3.31E-002\\
				0.003731&0.006260&6.80E-003&8.58E-003&2.41E-002\\
				\hline
		\end{tabular}\caption{Example~\ref{ex:hj}: errors for the \UB~ scheme with initial condition \eqref{eq:smooth} at time $T=0.5$.}
		\label{tab:hj_ub_smooth}
	\end{table}
	\begin{table}[!hbt]
				\captionsetup{width=1\textwidth}
\begin{tabular}{|c|c|c|c|c|}
				\hline
				$\Delta t$ & $\Delta x$ &$L^1$ Error & $L^2$ Error& $L^\infty$ Error \\
				\hline
				0.125000&0.210526&3.11E-002&2.62E-002&2.70E-002\\ 
				0.055556&0.102564&2.27E-002&2.08E-002&2.52E-002 \\ 
				0.029412&0.050633&1.10E-002&1.01E-002&1.26E-002    \\ 
				0.014706&0.025157&5.96E-003&5.75E-003&7.45E-003 \\
				0.007463&0.012539&2.93E-003&2.85E-003&3.72E-003\\
				0.003731&0.006260&1.47E-003&1.44E-003&1.89E-003 \\
				\hline
		\end{tabular}\caption{Example~\ref{ex:hj}: errors for the \SL  scheme with initial condition~\eqref{eq:smooth} at time $T=0.5$.}
		\label{tab:hj_sl_smooth}
	\end{table}
	\begin{table}[!hbt]
				\captionsetup{width=1\textwidth}
		\begin{tabular}{|c|c|c|c|c|}
				\hline
				$\Delta t$ & $\Delta x$ &$L^1$ Error & $L^2$ Error& $L^\infty$ Error \\
				\hline
				0.125000&0.210526&5.01E-002&4.16E-002&4.27E-002\\ 
				0.055556&0.102564&2.59E-002&2.35E-002&2.90E-002\\ 
				0.029412&0.050633&1.16E-002&1.08E-002&1.36E-002  \\ 
				0.007463&0.012539&5.99E-003&5.71E-003&7.34E-003\\
				0.014706&0.025157&2.98E-003&2.90E-003&3.79E-003 \\
				0.003731&0.006260&1.49E-003&1.46E-003&1.90E-003\\
				\hline
		\end{tabular}\caption{Example~\ref{ex:hj}: errors for the coupled \SL + \UB~ scheme with initial condition~\eqref{eq:smooth} at time $T=0.5$.}
		\label{tab:hj_cs_smooth}
	\end{table}
\end{center}	
\end{example}
\section{Conclusion and future work}
In this paper, we recall the semi-Lagrangian \cite{FF14}, ultra-bee \cite{BZ07} schemes and coupled algorithm from \cite{S2018} for solving advection and Hamilton-Jacobi equations. The scheme uses a switching indicator to decide which method to apply in each cell, depending on the smoothness and stability of the solution. We focus on the combination of an \SL~and \UB~method, and we show that the scheme can capture jumps and singularities accurately. The same idea of the scheme can be applied to other methods and can be simplified when the methods share the same grid and the same type of values, because then we do not need to project the values between different grids. We analyse the properties of the coupled scheme for the advection problem and we hope that they can be extended to non-linear Hamilton-Jacobi equations, as suggested by the last example. We plan to study this extension of the properties for HJ equation and the generalisation to 2D problems in future work.
\section*{Acknowledgements.}We dedicate this paper to my late PhD supervisor, Professor Maurizio Falcone, we started this work together few years ago. He was a great mentor. I am honoured to have worked with him and learned from him. His legacy will live on through his publications and many students.
\bibliographystyle{plainnat}

%
%

\section*{Declarations}
\subsection*{Ethical Approval }
Not applicable. 
\subsection*{Competing interests }
The authors declare that they have no known competing financial interests or personal relationships that could have appeared to influence the work reported in this paper.
\subsection*{Author's contributions }
I am the only author, so I am responsible for this paper.
\subsection*{Funding }
Not applicable.
\subsection*{Availability of data and materials}
All data supporting the findings of this study are presented within this paper. There are no additional datasets associated with this research. Readers can find all relevant data, figures, and tables in the main body of the paper.

\end{document}